\documentclass[12pt,reqno]{article}
\usepackage{fullpage} 
\usepackage{amssymb} 
\usepackage{graphicx} 
\usepackage{amsmath} 
\usepackage{amsthm,amssymb,latexsym} 
\usepackage{amstext, amsfonts,amsbsy} 
\usepackage{enumerate} 
\usepackage{upgreek} 
\usepackage{color} 
\usepackage{mathtools} 
\usepackage{tikz}
\usetikzlibrary{matrix,graphs,arrows,positioning,calc,decorations.markings,shapes.symbols}
\usetikzlibrary{cd}
 \usepackage{tikz-cd}
 \usepackage{float}
 
 \usepackage[english]{babel}
 \usepackage[pdftex,bookmarks,colorlinks,breaklinks]{hyperref}  
 \addto\extrasenglish{
}
 \addto\extrasenglish{
}
 \addto\extrasenglish{
}

\definecolor{dullmagenta}{rgb}{0.4,0,0.4}   
\definecolor{darkblue}{rgb}{0,0,0.4}
\hypersetup{linkcolor=red,citecolor=blue,filecolor=dullmagenta,urlcolor=darkblue} 
\usepackage{eucal}
\usepackage[title]{appendix}

\newtheorem{theorem}{Theorem}[section]
\newtheorem{proposition}[theorem]{Proposition}

\newtheorem{lemma}[theorem]{Lemma}
\newtheorem{definition}[theorem]{Definition}

\newcommand{\Pain}[1]{\text{P}_{\mathrm{#1}}}
\newcommand{\dPain}[1]{\text{P}\left(\mathrm{#1}\right)}

\newcommand{\C}{\mathbb{C}}
\newcommand{\Z}{\mathbb{Z}}

\newcommand{\p}{\mathbb{P}}

\newcommand{\X}{\mathcal{X}}

\newcommand{\h}{\mathcal{H}}
\newcommand{\E}{\mathcal{E}}
\newcommand{\F}{\mathcal{F}}
\newcommand{\Pic}{\operatorname{Pic}}

\newcommand{\Div}{\operatorname{Div}}
\newcommand{\Cl}{\operatorname{Cl}}
\newcommand{\pain}[1]{\text{P}_{\mathrm{#1}}}

\theoremstyle{definition}

\theoremstyle{remark}
\newtheorem{remark}[theorem]{Remark}

\numberwithin{equation}{section}

\begin{document}

{\noindent\Large\bf Differential equations for the recurrence coefficients of semi-classical orthogonal polynomials and their relation to the Painlev\'e equations via the geometric approach 
}
\medskip
\begin{flushleft}

\textbf{Anton Dzhamay}\\
School of Mathematical Sciences, The University of Northern Colorado, Greeley, CO 80526, USA\\
E-mail: \href{mailto:adzham@unco.edu}{\texttt{adzham@unco.edu}}\\[5pt]

\textbf{Galina Filipuk}\\
Faculty of Mathematics, Informatics and Mechanics, University of Warsaw, Banacha 2, Warsaw, 02--097, Poland\\
E-mail: \href{mailto:filipuk@mimuw.edu.pl}{\texttt{filipuk@mimuw.edu.pl}}\\[5pt] 

\textbf{Alexander Stokes}\\
Graduate School of Mathematical Sciences, The University of Tokyo, 3--8--1 Komaba Meguro-ku Tokyo, 153--8914, Japan\\
E-mail: \href{mailto:alexander.stokes.14@ucl.ac.uk}{\texttt{alexander.stokes.14@ucl.ac.uk}}\\[8pt]

\\[3pt]


\end{flushleft}

\date{\today}

\begin{abstract}

In this paper we present a general scheme for how to relate differential equations for the recurrence coefficients of semi-classical orthogonal polynomials to the Painlev\'e equations using the geometric framework of the 
Okamoto Space of Initial Conditions.
We demonstrate this procedure in two examples. 
For semi-classical Laguerre polynomials appearing in \cite{HC17},  we show how the recurrence coefficients are connected to the fourth Painlev\'e equation. 
For discrete orthogonal polynomials associated with the hypergeometric weight appearing in \cite{FVA18} we  discuss the relation of the recurrence coefficients to the sixth Painlev\'e equation, 
extending the results of \cite{DFS19}, where a similar approach was used for a discrete system for the same recurrence coefficients. Though the discrete and differential systems here share the same geometry, the construction of the space of initial conditions from the differential system is different and reveals extra considerations that must be made.
We also discuss  a number of related topics in the context of the geometric approach, such as Hamiltonian forms of the differential equations for the recurrence coefficients, Riccati solutions for special parameter values, and associated discrete Painlev\'e equations.

\end{abstract}

\section{Introduction} 
\label{sec:introduction}

It is well-known that a sequence of orthonormal   polynomials  $p_{n}(x) = \gamma_{n} x^{n} + \dots +\gamma_1 x + \gamma_0$ indexed by their degree $n$ satisfy a three term recurrence relation
\begin{equation}
	xp_n(x) = a_{n+1}p_{n+1}(x) + b_n p_n(x) + a_n p_{n-1}(x),
\end{equation}
where  $a_0=0$. The coefficients $a_n$ and $b_n$ are usually refered to as the \emph{recurrence coefficients} \cite{Chi:1978:AITOP, Ism:2005:CAQOPIOV, Sze:1967:OP}. 
The corresponding monic orthogonal polynomials $P_n = p_n/\gamma_n$ satisfy a similar three term recurrence relation
\begin{equation}
	xP_n(x) = P_{n+1}(x) + b_n P_n(x) + a_n^2 P_{n-1}(x). 
\end{equation}
Recently there has been considerable interest in connections between recurrence coefficients of semi-classical orthogonal polynomials and solutions of discrete or differential Painlev\'e equations (see, for instance, \cite{Van:2018:OPAPE} and numerous references therein).  
Usually it is first shown that the recurrence coefficients $\{a_{n},b_{n}\}$, as functions of the discrete variable $n$, satisfy, after some 
change of variables, a system of non-linear difference equations and, as functions also
of some continuous parameter appearing in the weight, satisfy a Toda-type differential-difference system. 
From these systems one can obtain a scalar second order (first or higher degree) nonlinear differential equation, which is usually very cumbersome. At this point the following questions arise, which we address in this paper: Is this differential equation (or some equivalent system of first order differential equations) reducible to one of the  Painlev\'e equations? If so, which Painlev\'e equation? How may one find an explicit change of variables realising this? 

In \cite{DFS19} the authors considered similar questions for discrete equations, and presented a general framework for how to determine whether a given discrete system can be reduced to a standard discrete Painlev\'e equation and, if so, obtain the relation explicitly. 
The main tool was the geometric theory of discrete Painlev\'e equations developed by H. Sakai in the seminal paper \cite{Sak:2001:RSAWARSGPE} (see also \cite{KajNouYam:2017:GAOPE}).  
Sakai's theory was preceded by the fundamental work of K. Okamoto \cite{Ok79}, in which he constructed spaces on which the differential Painlev\'e equations were, in a sense, regularised. 
The geometric approach provides powerful tools for the study of Painlev\'e equations, as evidenced by the large number of important studies that have taken cues from the work of Okamoto and Sakai, for instance \cite{
JD11, 
Sak13, ST02, T07, TOS05} in the differential case,   \cite{CDT17,DK19,DST13,DT18,Takenawa2001a,Takenawa2001b} in the discrete case, as well as references therein. 

In this paper we present a general procedure for answering the questions above in the differential case using a geometric approach. 
While the geometric theory is well-documented in the literature, this paper is intended as a self-contained guide to the relevant techniques for use by researchers to whom such identification problems are of interest, but for whom the geometric theory may be new. 
For this reason, this paper includes a detailed exposition of the background material and comprehensive illustrations of the calculation techniques, as well as data from \cite{KajNouYam:2017:GAOPE} and \cite{DFS19} necessary for some of the results.
While the examples we consider come from the theory of orthogonal polynomials, we emphasise that the procedure is applicable to any system of differential equations 
suspected to be reducible to one of the Painlev\'e equations.
We are essentially considering the question of which Painlev\'e equation a given second-order equation with the Painlev\'e property is equivalent to, which is sometimes referred to as the ``Painlev\'e  equivalence problem'' \cite{Cla19}. In cases where the equation can be written as a pair of first-order first-degree differential equations, our approach is applicable and solves the problem.

We have chosen two differential systems  as illustrative examples for this paper, with the  general procedure presented in such a way that anyone interested would be able to make necessary changes to adapt the calculations to another problem of interest. 
The first example is a system related to the recurrence coefficients of semi-classical Laguerre polynomials \cite[Cor. 2.5]{HC17}, which we will show is connected to the fourth Painlev\'e equation and transform it to the standard form. 
This system was deduced but not written down explicitly in \cite{HFC19}, so we present it below. 
The second is a system related to the recurrence coefficients of discrete orthogonal polynomials with the hypergeometric weight defined in \cite{FVA18}. 
 For this weight, although it is known that the recurrence coefficients are related to the sixth Painlev\'e equation \cite{FVA18} (see also \cite{HFC19} for an approach by direct calculation), the geometric method reveals some new features of the differential system and its regularisation, which we wanted to discuss in detail. 
In addition, the analysis of the hypergeometric case makes use of a significant amount of data from \cite{DFS19}, in which the authors applied a similar geometric method to a system of difference equations for the recurrence coefficients.
The underlying geometry is the same for the differential equations, but its construction from the system itself is different and requires some careful considerations that we describe here.

\subsection{Background}
Solutions of nonlinear differential equations may have singularities that depend not just on the equation, but also on
initial conditions.  
P. Painlev\'e  defined a property of nonlinear ordinary differential equations 
(essentially that solutions are single-valued about \emph{movable} singularities, i.e., those whose locations are dependent on initial conditions), which  is now known as the \emph{Painlev\'e property}. 
P. Painlev\'e and his student B.~Gambier then studied a large class of  second-order nonlinear  differential equations that satisfy this condition and found that all the families could be solved in terms of elementary functions, classical special functions, as well as the solutions of six new families,
now known as the  \emph{Painlev\'e equations} $\Pain{I},\dots, \Pain{VI}$. 
Solutions of these equations, the so-called \emph{Painlev\'e transcendents}, 
are indeed new \emph{purely nonlinear special functions}. 
Over the last fifty years Painlev\'e transcendents have been playing an 
increasingly important role in many nonlinear models in mathematics and physics, from Quantum Cohomology to the theory of Random Matrices. 
Probably the most important example is the famous Tracy-Widom distribution from Random Matrix Theory, which can be expressed in terms of 
the Hastings-McLeod solution of $\Pain{II}$ \cite{TracyWidomA, TracyWidomB}. 

\subsubsection{Okamoto's space} 
The Painlev\'e equations possess a number of remarkable properties, for example B\"acklund transformation symmetries (which relate solutions with various values of the parameters), Hamiltonian forms, Lax pairs, classical solutions for special parameter values, among many more \cite{GLS, IKSY, No, NouYam:1998:AWGDDSAPE}. 
In particular, Hamiltonian forms and symmetries of the Painlev\'e equations were studied in \cite{OkI, OkII, Ok1, Ok2, Ok3, Ok4, OKSO}. \\
Okamoto also discovered a geometric structure common to the six Painlev\'e equations, namely the existence of a \emph{space of initial conditions} \cite{Ok79}. For each $\Pain{J}$, Okamoto considered an equivalent non-autonomous Hamiltonian system with polynomial Hamiltonian on the trivial bundle $\C^2 \times B_{\operatorname{J}}$ over the independent variable space $B_{\operatorname{J}}$ (the complex plane with fixed singularities of the equation removed).
By first compactifying the $\C^2$-fibers, blowing up certain singularities then removing certain curves, Okamoto obtained a  bundle $E_{\operatorname{J}}$ over $B_{\operatorname{J}}$, of which the flow of the Hamiltonian system induces a uniform foliation. 
Each fiber parametrises the set of solutions of $\Pain{J}$, so can be regarded as a \emph{space of initial conditions}. 
These spaces were studied further in a series of papers \cite{T1, T2, M} for the Hamiltonian forms of the Painlev\'e equations $\Pain{II}$-$\Pain{VI}$ (the $\Pain{I}$ case was later considered in \cite{orbifold}, see also \cite{chiba}). 
In particular, for each space $E_{\operatorname{J}}$ the authors constructed a \emph{symplectic atlas}, provided by a number of coordinate neighbourhoods between which the transition maps are rational and symplectic. 
Moreover, certain uniqueness results for Hamiltonian systems on these spaces were proved:   any non-autonomous Hamiltonian system whose Hamiltonian structure is  holomorphic on $E_{\operatorname{J}}$ and 
meromorphically extendable to its closure must coincide with Okamoto's Hamiltonian form of the Painlev\'e equation $\Pain{J}$.  
This gives rise to the idea that a global analysis of the Hamiltonian forms of the Painlev\'e equations reduces to the geometry of these spaces of initial conditions.

\subsubsection{Semi-classical Laguerre weight} \label{Laguerresubsubsection}

We consider a semi-classical Laguerre weight given by $w(x,c)=w(x,\alpha,c):=x^{\alpha} e^{-N(x+c(x^2 -x))}$, where $x\in (0,\infty)$, $\alpha>-1$, $c\in [0,1],$ $N>0$.  
A similar weight $\tilde{w}(x, \alpha, t) = x^{\alpha} e^{- x^2 + t x}$ was considered in \cite{BVA18, FVAZ12, CJ14}, where equations for the recurrence coefficients were derived and studied, with the differential equations in particular found to be related to $\pain{IV}$.

In \cite{HC17} it was shown that the recurrence coefficients satisfy
a system of  discrete equations
\begin{equation}\label{discrete1}
\begin{aligned} 
2c (a_{n+1}^2+a_n^2)+2c b_n^2+(1-c)b_n &=(2n+\alpha+1)/N  , \\
a_n^2(2c b_n+1-c)(2c b_{n-1}+1-c) &=  (2c a_n^2-n/N)(2c a_n^2-(n+\alpha)/N),
\end{aligned}
\end{equation}
with initial conditions $a_0^2=0$ and $b_0$ as an expression in terms of the parabolic cylinder functions. 
Moreover, as functions of the parameter $c$ in the weight they satisfy a Toda-type system
\begin{equation}\label{Toda P4}
\begin{aligned} 
2c {(\rm{ln}} (a_n^2) )'   &=N(1+c)(b_n-b_{n-1}-2)  , \\
2c b_n'  & =-b_n+N(1+c)(a_{n+1}^2 -a_n^2)  .
\end{aligned}
\end{equation}
Note that we have made slight changes to the notation in \cite{HC17}, renaming the parameter $s$ by $c$ and replacing the functions $R_n(s)$  and $r_n(s)$ by $x_n(c)$ and $y_n(c)$ respectively.   The recurrence coefficients of the monic orthogonal polynomials with this semi-classical Laguerre weight   are related to $x_n$ and $y_n$ as follows: 
$$b_n=(c-1+x_n)/(2c), \;\;a_n^2=(n+N y_n)/(2N c).$$ 
With this notation the discrete system (\ref{discrete1}) becomes
\begin{equation} \label{discrete}
\begin{aligned}
x_n x_{n-1} &= \frac{2Nc y_n(y_n-\alpha/N) }{n+N y_n} , \\
2(y_{n+1}+y_n) &= \frac{2\alpha}{N}-\frac{x_n^2+(c-1)x_n}{c} .
\end{aligned}
\end{equation}
From (\ref{Toda P4}) and (\ref{discrete}) one may obtain the following system of differential equations for  $x(c):=x_n(c),\, y(c):=y_n(c) $: 
\begin{equation} \label{differential}
\left\{ 
\begin{aligned} 
\frac{x'}{c+1} &= -N \frac{x^2}{4c^2} + \left( \frac{N(1-c)}{4c^2} + \frac{1}{2c(c+1)} \right) x - N \frac{y}{c} + \frac{\alpha}{2c}, \\
\frac{y'}{c+1} &= -N \frac{y^2}{2c x} + \left( \frac{\alpha}{2 c x} + N \frac{x}{4c^2} \right) y  + n \frac{x}{4 c^2} .
\end{aligned}
\right.
\end{equation}
System (\ref{differential}) is a system of coupled Riccati equations for $x(c)$ and $y(c)$, which is the differential system of interest to us in this case. Eliminating $y$, one may obtain a  second order ordinary differential equation (ODE), which the authors in \cite{HC17} remarked was likely to be equivalent to one of the Painlev\'e equations, but did not transform it to any of the standard forms. 
In addition to the fact that the recurrence coefficients for the similar but less general weight $\tilde{w}(x,\alpha,t) = x^{\alpha} e^{- x^2 + t x}$ are known to be related to $\pain{IV}$, there are several reasons that this is the Painlev\'e equation to which one should suspect the system \eqref{differential} is related, which we outline now\footnotemark.
\footnotetext{ We thank an anonymous reviewer of an earlier version of this paper for bringing these observations to our attention.}

Firstly, the second-order equation for $x(c)$ is given explicitly by
\begin{equation}
\begin{aligned}
x'' &= \frac{1}{2 x} (x')^2 - \frac{1}{c(1+c)} x' + \frac{3 N^2 (1+c)^2 }{32 c^4 }x^3 + \frac{N^2 (c-1)(c+1)^2 }{8 c^2} x^2 \\
&\qquad \left[ \frac{ (1-c^2)^2 N^2}{32 c^4} - \frac{ N (2n + \alpha +1)(1+c)^2}{8 c^3} + \frac{1-3c}{8 c^2(1+c)} \right] x - \frac{\alpha^2 ( 1+c)^2}{8 c^2 x},
\end{aligned}
\end{equation}
which we see has a similar form to the usual scalar fourth Painlev\'e equation
\begin{equation}
\frac{d^2 w}{dz^2} = \frac{1}{2w} \left( \frac{dw}{dz} \right)^2 + \frac{3}{2} w^3 + 4 z w^2 + 2(z^2 - A) w + \frac{B}{w},
\end{equation}
where $A, B$ are constants. 
In addition, the fact that the initial conditions for the discrete system are given in terms of the parabolic cylinder functions also suggests a connection to $\pain{IV}$, since these initial conditions will correspond to a special solution of the differential system \eqref{differential} at $n=0$. Under a transformation to one of the Painlev\'e equations these (and indeed the whole sequence of recurrence coefficients) would be mapped to a hierarchy of special solutions. The fact that $\pain{IV}$ is the Painlev\'e equation which admits special solutions expressed in terms of parabolic cylinder functions is further indication that this is the natural ``target equation" to which one should seek a transformation from the system \eqref{differential}. We remark as well that when $c=1$ the problem can be reduced to the study of the weight $\tilde{w}(x,\alpha,1) = x^{\alpha} e^{- x^2 + x}$, for which the connection to the parabolic cylinder functions and $\pain{IV}$ were already identified in \cite{CJ14}. This weight was also studied in \cite{FVAZ12}.
 

We show step by step how we can identify the connection with the fourth Painlev\'e equation and find an explicit change of variables to the standard form. 
While this can be achieved by alternative methods, in particular identifying $\pain{IV}$ as the target equation as outlined above and proceeding to isolate the transformation by direct computation, this relies on several choices being made, notably the form of the Ansatz for the transformation and the matching of parameters.
The geometric method is systematic and eliminates the need for these choices.
The purpose of this paper is to present this method itself as well as extra insights it yields, such as Hamiltonian structures for the differential equations for the recurrence coefficients.
We choose this example in order to demonstrate this method in a detailed and systematic way, with full exposition of the required geometric theory so it can be adapted to other cases.

\subsubsection{Discrete orthogonal polynomials with hypergeometric weight}
The discrete orthogonal polynomials $p_n(x)$ with the hypergeometric weight are defined as follows \cite{FVA18}: they  are orthonormal polynomials on the set $\mathbb{N}=\{0,1,2,\ldots\}$ of non-negative integers with respect to the hypergeometric weight $w_{k}$, so
\begin{equation}\label{eq:hyp-weight}
	\sum_{k=0}^\infty p_n(k)p_m(k) w_k = \delta_{m,n}, \qquad 
	w_k = \frac{(\alpha)_k (\beta)_k}{(\gamma)_k k!} c^k, \quad \alpha,\beta,\gamma >0,\ 0 < c < 1,
\end{equation}
where $(\cdot)_k$ is the usual Pochhammer symbol and $\delta_{m,n}$ is the Kronecker delta.    In \cite[Theorem 3.1]{FVA18} a system of two first-order difference equations was obtained for variables $x_n$ and $y_n$ related to the recurrence coefficients $a_n^2$ and $b_n$ as follows:
\begin{equation}\label{(2.1)}
   a_n^2\frac{1-c}{c}=y_n+\sum_{k=0}^{n-1}x_{k}+\frac{n(n+\alpha+\beta-\gamma-1)}{1-c},
\end{equation}
\begin{equation}\label{(2.2)}
   b_{n}=x_{n}+\frac{n+(n+\alpha+\beta)c-\gamma}{1-c}.
\end{equation}
In addition, we have that
\begin{equation}\label{(2.5)}
    \frac{(1-c)^2}{c^2} a_n^2x_nx_{n-1}=y_n(y_n-\alpha\beta+\frac{\gamma}{c})+(\alpha\beta-y_n)\frac{1-c}{c}\sum_{k=0}^{n-1}x_{k}.
\end{equation}
From (\ref{(2.1)}) and (\ref{(2.5)}) one may obtain an alternative expression for $a_n^2$:
\begin{equation}\label{(2.6)}
    a_n^2=\frac{n\alpha\beta c(n+\alpha+\beta-\gamma-1)-c[n^2+n(\alpha+\beta-\gamma-1)-\alpha\beta+\gamma]y_n-cy_n^2}{(c-1)^2(\alpha\beta-x_{n-1}x_n-y_n)}.
\end{equation}
The sequences $(x_n)_{n\in\mathbb{N}},\,(y_n)_{n\in\mathbb{N}}$ satisfy the following discrete system  (see \cite[Theorem 3.1]{FVA18}):
\begin{gather}\label{(2.3)}
\begin{split}
(y_{n}-&\alpha\beta+(\alpha+\beta+n)x_{n}-x_{n}^{2})(y_{n+1}-\alpha\beta+(\alpha+\beta+n+1)x_{n}-x_{n}^{2})\\
       &=\frac{1}{c}(x_{n}-1)(x_{n}-\alpha)(x_{n}-\beta)(x_{n}-\gamma),
\end{split}
\end{gather}
and
\begin{gather}\label{(2.4)}
\begin{split}
&(x_{n}+Y_{n})(x_{n-1}+Y_{n})\\
&=\frac{(y_{n}+n\alpha)(y_{n}+n\beta)(y_{n}+n\gamma-(\gamma-\alpha)(\gamma-\beta))(y_{n}+n-(1-\alpha)(1-\beta))}{(y_{n}(2n+\alpha+\beta-\gamma-1)+n((n+\alpha+\beta)(n+\alpha+\beta-\gamma-1)-\alpha\beta+\gamma))^{2}},
\end{split}
\end{gather}
where
\begin{equation*}
Y_{n}=\frac{y_{n}^{2}+y_{n}(n(n+\alpha+\beta-\gamma-1)-\alpha\beta+\gamma)-\alpha\beta n(n+\alpha+\beta-\gamma-1)}{y_{n}(2n+\alpha+\beta-\gamma-1)+n((n+\alpha+\beta)(n+\alpha+\beta-\gamma-1)-\alpha\beta+\gamma)}.
\end{equation*}
The initial values $x_0$ and $y_0$ are given by
\begin{equation}\label{initial conds}
x_0=\frac{\alpha\beta c}{\gamma}\frac{{}_2F_{1}(\alpha+1,\beta+1;\gamma+1;c)}{{}_2F_{1}(\alpha,\beta;\gamma;c)}+\frac{(\alpha+\beta)c-\gamma}{c-1},\;\; y_0=0,
\end{equation}
where $_2F_1$ is the Gauss hypergeometric function. For the hypergeometric weight the recurrence coefficients satisfy the Toda system
\begin{eqnarray}\label{(3.1)}
     c \frac{d}{dc} a_n^2 &=& a_n^2(b_n-b_{n-1}), \qquad n \geq 1,  \\ \label{(3.2)}
     c \frac{d}{dc} b_n &=&  a_{n+1}^2 - a_n^2, \qquad n \geq 0.
\end{eqnarray}

Let us recall, from \cite{HFC19}, the procedure for obtaining a system of differential equations for $x_n(c)$ and $y_n(c)$.   First replace $x_n$ and $y_n$ in (\ref{(2.3)}) and (\ref{(2.4)}) by $x_n(c)$ and $y_n(c)$  respectively.   We may solve equation (\ref{(2.3)}) for $y_{n+1}(c)$ in terms of  $x_n(c)$ and $y_n(c)$, and similarly we may solve equation (\ref{(2.4)}) to give $x_{n-1}(c)$ in terms of $x_n(c)$ and $y_n(c)$. Next we replace $n$ in  (\ref{(2.4)}) by $n+1$ and substitute the expression for $y_{n+1}(c)$ found previously. This gives us an opportunity to find an expression for $x_{n+1}(c)$ in terms of $x_n(c)$ and $y_n(c)$. Next, we need to modify the Toda system. Equations (\ref{(2.2)}) and (\ref{(2.6)}) are expressions for the recurrence coefficients $a_n^2(c)$ and $b_n(c)$  in terms of $x_n(c)$, $ x_{n-1}(c)$ and $y_n(c)$. We substitute (\ref{(2.2)}) and (\ref{(2.6)}) into the Toda system. Next we substitute $y_{n+1}(c)$, $x_{n-1}(c)$ and their derivatives along with the expression for $x_{n+1}(c)$ found previously into the modified Toda system. This gives us a system of two first order differential equations  for $x(c)=x_n(c)$ and $y(c)=y_n(c)$ of the form
\begin{equation}\label{syst P6}
\left\{
\begin{aligned}
x'(c)&=\frac{P_1(x(c),y(c),c)}{Q(x(c),y(c),c)},\\
y'(c)&=\frac{P_2(x(c),y(c),c)}{Q(x(c),y(c),c)},
\end{aligned}
\right.
\end{equation}
 where $P_1$, $P_2$ and $Q$ are polynomials in their arguments. 
 Explicitly, with $x=x(c)$, $y=y(c)$, these are given by
\begin{equation*}
\begin{aligned}
P_1(x,y,c) &= (1-c)x^4 + [-\alpha -\beta +2 c (\alpha +\beta +n)-\gamma -1] x^3 \\
&\quad+\left[\alpha  (\beta +\gamma +1)+\beta  \gamma +\beta -c \left(\beta ^2+2 \beta  (2 \alpha +n)+(\alpha +n)^2\right)+\gamma \right] x^2 \\
&\quad+ \left[ \alpha  \beta  (2 c (\alpha +\beta +n)-1)-\gamma  (\alpha  \beta +\alpha +\beta ) \right] x + \alpha  \beta  (\gamma -\alpha  \beta  c) \\
&\quad+ 2c y\left[ x^2 -(\alpha+\beta+n) x + \alpha \beta \right] - c y^2,\\
P_2(x,y,c) &= n \left[\alpha ^2+\alpha  \beta -\alpha +\beta ^2-\beta +\gamma +n^2-\gamma  (\alpha +\beta +n)+2 \alpha  n+2 \beta  n-n\right] x^2 \\
&\quad - 2 \alpha  \beta  n \left(\alpha +\beta -\gamma +n-1 \right) x + n \alpha \beta(\alpha \beta - \gamma) \\
&\quad+y \Big[  \left(\alpha +\beta -\gamma +2 n-1\right) x^2 + 2 \left(-\alpha  \beta +\gamma +n^2+n (\alpha +\beta -\gamma -1) \right) x \\
&\quad  -\alpha  \gamma -\beta  \gamma +\alpha  \beta  (\gamma -2 n+1) \Big] + y^2 \left[ 2x - \gamma +n -1 \right],\\
Q(x,y,c)&=c(c-1)(\alpha\beta-(n+\alpha+\beta)x+x^2-y).
\end{aligned}
\end{equation*}
This system (\ref{syst P6}) is the one of interest in the case of the hypergeometric weight, and we analyse it below in detail.

\subsection{Outline of the paper}

We will begin by outlining the method as a step-by-step procedure in \autoref{section2}, then demonstrate it in detail in \autoref{section3} for the first example of the system (\ref{differential}) from the semi-classical Laguerre weight, then in \autoref{section4} for the system (\ref{syst P6}) from the hypergeometric weight. In the process we will point out some aspects of the geometric theory of Painlev\'e equations that are important to our approach as well as some features of the equations that are uncovered through our analysis. 

Relating to the method itself, we note some important considerations that must be made in using blowups to construct spaces of initial conditions for differential systems, which are different from the discrete case. In particular careful attention that has to be paid to the behaviour of the vector field defining the system, rather than simply where its components have indeterminacies.
Further, once a space of initial conditions is constructed and found to correspond to a family of Sakai surfaces, it is slightly simpler in the differential case to obtain the transformation to the standard Painlev\'e equation as we do not have to ensure the identification on the level of the Picard lattice matches the translations giving the discrete dynamics,
see Remark 3.11 (unless there is a particular desired matching of parameters between the differential systems as in for example \cite{DFLS}). 
There is also a key difference in the theoretical basis for our procedures for identifying systems as equivalent to Painlev\'e equations between the discrete and differential cases: for discrete systems, if we construct a space of initial conditions, then the system is discrete Painlev\'e by definition if this is a family of Sakai surfaces and the dynamics correspond to a translation symmetry. 
For differential systems, after constructing a space of initial conditions given by a family of Sakai surfaces of one of the types corresponding to the differential Painlev\'e equations, we need an extra step to justify the fact that the identification of surfaces transforms the given system to the standard Painlev\'e equation, via the uniqueness results for Hamiltonian systems on Okamoto's spaces \cite{T1, T2, M}. 

Relating to the features of the differential systems for the recurrence coefficients beyond the identifications with the standard forms of Painlev\'e equations, for both examples we obtain Hamiltonian structures for the differential equations using geometric techniques.
We also show how the initial conditions for the discrete systems in each case correspond to seed solutions of a hierarchy of Riccati special solutions of the differential Painlev\'e equation, so the recurrence coefficients can in principal be written explicitly using the determinantal representations for this hierarchy.
%
%

\section{The identification procedure for differential systems} \label{section2}

This procedure consists of the following steps, where
we assume that we indeed can find a birational transformation to some differential Painev\'e equation, otherwise the process will terminate at some step. Note that the procedure follows some of the same lines as the discrete case presented in \cite{DFS19}, in particular Steps 2-4.
However there are certain differences, particularly in the construction of a space of initial conditions, which we will point out when we arrive at them. 
In the following outline we make reference to some concepts from the geometric theory, which will be explained in detail in our first expository example.

\begin{enumerate}[(Step 1)]
	\item \textbf{Construct a space of initial conditions for the system.} 
			Begin by considering the system as a pair of first order equations. We compactify the fibres of the phase space of the nonautonomous system from $\mathbb{C}^{2} = \mathbb{C}\times \mathbb{C}$ to $\mathbb{P}^{1} \times \mathbb{P}^{1}$, so we have a trivial bundle over $\C$ (possibly with some fixed singularities of the system removed) on which the differential equation corresponds to a rational vector field. Find the points where the components of the vector field have indeterminacies (points where
both the numerator and the denominator of the rational functions vanish). Here we run into a subtle point that is particular to the differential case, which is that we must determine whether these singular points of the vector field are accessible (in the sense that they can be reached by solutions with initial data where the vector field is regular). Resolve all such accessible singular points using the blowup procedure, until we have a space to which the vector field lifts to one with no more accessible singularities (we will address the question of when an indeterminacy of the vector field represents an accessible singularity in detail in our second example). Identify the \emph{inaccessible divisors} in the fibers of the resulting bundle and remove them, after which we arrive at a space of initial conditions: a bundle over the independent variable space admitting a uniform foliation by solution curves transverse to the fibers.  
	\item \textbf{Determine the surface type, according to Sakai's classification scheme.} 
            In Sakai's scheme, the surfaces associated with Painlev\'e equations are obtained from $\p^1\times\p^1$ through the blowups of \emph{eight} points. However in practice we may require more blowups to arrive at a space of initial conditions for the differential system, which would result in a non-minimal surface which must be further transformed by blowing down some $(-1)$-curves to arrive at a Sakai surface.
			Once we have a space of initial conditions, the inaccessible divisors on each surface should form the irreducible components of an anticanonical divisor and should each have self-intersection index $-2$. This divisor can be identified as the proper transform of 
			some 
			biquadratic curve on $\mathbb{P}^{1}\times \mathbb{P}^{1}$ (i.e., a curve whose defining polynomial, when written in a coordinate chart, has
			bi-degree $(2,2)$). The components of the anticanonical divisor should also be associated with an \emph{affine Dynkin diagram}; components correspond to nodes, which are connected when the corresponding components intersect. The type of this Dynkin diagram determines the type of the surface in Sakai's classification, and is called the \emph{surface type} of the system. 
	\item \textbf{Find an identification with the standard model on the level of $\operatorname{Pic}(\mathcal{X})$.} At this step, make an identification between the surfaces obtained above and the standard model of Sakai surfaces of the same type on the level of their Picard lattices. We need to ensure that this change of basis 
			identifies the \emph{surface roots} (or nodes of the Dynkin diagrams of our surface corresponding to the inaccessible divisors) with the standard example, and matches the semigroups of effective divisor classes. Note that, 
			unless there is some desired parameter matching with the standard form of the Painlev\'e equation in question,
the method from this point onward is simpler than that for the discrete case \cite{DFS19}. 
If we are interested in identifying only the differential systems, rather than the discrete and differential systems simultaneously, then we do not need to adjust this identification.
	\item \textbf{Find the change of variables reducing the given system to the standard form.}  We need to find the birational mapping that induces
			the identification from the previous step, which will provide the change of variables to the standard form of the relevant Painlev\'e equation. For this, we form an Ansatz and successively impose conditions from the identification to determine the coefficients. An important part
			of this computation is the identification of various parameters between the two problems, which may be done using the \emph{period mapping} on the surfaces forming our space of initial conditions.
	
\end{enumerate}

We remark 
that the compactification in Step 1 may be chosen as a different minimal model, namely $\p^2$ or one of the Hirzebruch surfaces $\mathbb{F}_l$.
A different choice of compactification does not in principle pose a problem, and the only adaptations to the method to be made come from elementary facts about these minimal models, e.g. the structure of their Picard groups.
However, in our experience choosing one of the Hirzebruch surfaces leads more often to situations where blowdowns are required to arrive at a family of Sakai surfaces, and we make the choice of $\p^1 \times \p^1$ since this is what is used in the reference models of surfaces in \cite{KajNouYam:2017:GAOPE}.

\section{Semi-classical Laguerre weight} \label{section3}

\subsection{The space of initial conditions for the semi-classical Laguerre weight} \label{section31}
We first construct a space of initial conditions for the differential system \eqref{differential}. Since it is nonautonomous and has a fixed singularity at $c=0$, we consider its phase space first as a trivial bundle over $\C \backslash\{0\}$ with fiber over $c$ being $\C^2$ with coordinates $(x,y)$. We then compactify the $\C^2$-fibers to $\p^1 \times \p^1$. Considering the variables $(x,y)$ from \eqref{differential} as a set of affine coordinates, we introduce $X = 1/x, Y= 1/y$, so $\p^1 \times \p^1$ is covered by the usual four charts, namely $(x,y), (X,y), (x,Y), (X,Y)$. The equations \eqref{differential} then give a rational vector field on the part of the bundle visible in the $(x,y)$-chart. Via the transition functions $X = 1/x, Y=1/y$, this extends uniquely to define a rational vector field on the whole bundle, the indeterminacies of which we will be interested in. 

We see from \eqref{differential} that the vector field is regular on the part of the $(x,y)$-coordinate neighbourhood where $x \neq 0$.  The vector field diverges where $x = 0$ except for at the points of indeterminacy
\begin{equation}
q_1 : (x,y) = (0,0), \quad \quad q_2 : (x,y) = (0, \alpha / N),
\end{equation}
with both components $x',y'$ being indeterminate at both of these points. We blow up each of these points in the $\p^1\times\p^1$ fiber, introducing for each $q_i$ a pair of $\C^2$-coordinate charts $(u_i, v_i), (U_i, V_i)$, in which the \emph{exceptional divisor} $F_i$ replacing $q_i$ is given by $v_i=0,V_i=0$, respectively:
\begin{gather*} 
(x,y) = (u_1 v_1, v_1) = (V_1 , U_1 V_1), \\
(x,y - \alpha / N) = (u_2 v_2, v_2) = (V_2 , U_2 V_2) .
\end{gather*}
The vector field on the bundle lifts uniquely under the blowup of the fibre, and can be computed by direct substitution using the relations above as changes of variables. 
For example, in the $(u_1,v_1)$ chart we have
\begin{subequations} \label{u1v1}
\begin{align}
u_1' &= \frac{-2c(c+1)N + (2c+N-c^2N)u_1 - (c+1)n u_1^2 - 2N(c+1) u_1^2 v_1}{4 c^2}, \\
v_1' &= \frac{(c+1) \left( 2c \alpha - 2cN v_1 + n u_1^2 v_1 + N u_1^2 v_1^2 \right)}{4 c^2 u_1}.
\end{align}
\end{subequations}
Similarly, in the $(U_1 ,V_1)$ chart we have 
\begin{subequations} \label{U1V1}
\begin{align}
U_1' &= \frac{(c+1)n + (c^2N-N-2c) U_1 + 2c(c+1)N U_1^2 + 2(c+1)N U_1 V_1}{4 c^2},\\
V_1' &= \frac{2 c (c+1) \alpha + (2c+N-c^2N)V_1 - 4c(c+1)N U_1V_1 - (c+1)N V_1^2}{4 c^2}.
\end{align}
\end{subequations}
From \eqref{u1v1} and \eqref{U1V1}, we may deduce that the singularity at $q_1$ is resolved. To be precise, we have the vector field on the exceptional divisor being given in the first chart by substituting $v_1=0$ in \eqref{u1v1}:
\begin{subequations} \label{u1v1zero}
\begin{align}
u_1' &= \frac{-2c(c+1) N + (2c + N -c^2N)u_1 - (c+1)n u_1^2}{4c^2}, \\
v_1' &= \frac{(c+1) \alpha}{2c u_1},
\end{align}
\end{subequations}
and in the second chart by substituting $V_1 = 0$ in \eqref{U1V1}:
\begin{subequations} \label{U1V1zero}
\begin{align}
U_1' &= \frac{(c+1)n -(2c+N-c^2N)U_1+2c(c+1)N U_1^2 }{4c^2},\\
V_1' &= \frac{(c+1)\alpha}{2c}.
\end{align}
\end{subequations}
For each point $p$ on $F_1$ where $u_1 \neq 0$, the vector field is regular so for any path through $c_0 \in \C \backslash \{0\}$ we have a unique analytic solution passing through $p$ in the fiber over $c_0$. We already knew no solution curves pass through the part of the bundle given by $x=0$ away from $q_1, q_2$ in this chart. So, we have a family of disjoint local solution curves parametrised by where they intersect $F_1$, and we have resolved the singularity at $q_1$. Computing along the same lines allows us to deduce that the vector field is regular on the part of $F_2$ where $u_2 \neq 0$, and the singularity at $q_2$ is also resolved. We note that the points on $F_1, F_2$ given by $(u_1,v_1)=(0,0)$, $(u_2,v_2) = (0,0)$ respectively are inaccessible in the same sense that the part of the $(x,y)$-coordinate patch where $x=0$ but $y \neq 0,\alpha/N$ is. In fact these points on the exceptional lines correspond to their intersection with the \emph{proper transform} (also known as strict transform) of the line in $\p^1\times\p^1$ defined by $x=0$, which we explain now for completeness. 

Denote the projection map from the blowups of $q_1, q_2$ by 
\begin{equation}
\pi_{12} : \operatorname{Bl}_{q_1 q_2}(\p^1 \times \p^1) \rightarrow \p^1 \times \p^1,
\end{equation}
so $\pi_{12}$ is an isomorphism away from $q_1, q_2$, and the exceptional divisors are $F_1 = \pi_{12}^{-1}(q_1)$, $F_2 = \pi_{12}^{-1}(q_2)$. If $\mathcal{C}$ is a curve on $\p^1\times \p^1$, its \emph{total transform} (or pullback) is a curve on $\operatorname{Bl}_{q_1q_2}(\p^1 \times \p^1)$ given as the preimage $\pi_{12}^{-1} (\mathcal{C})$. If $\mathcal{C}$ is irreducible and does not pass through $q_1, q_2$, then its total transform will also be irreducible and isomorphic to $\mathcal{C}$. However, if $\mathcal{C}$ passes through either $q_1$ or $q_2$, then its preimage under $\pi_{12}$ will have  more irreducible components. For example, the line in $\p^1 \times\p^1$ defined by $x=0$ intersects both $q_1$ and $q_2$, and in the charts $(u_1, v_1)$, $(u_2, v_2)$ for $\operatorname{Bl}_{q_1q_2}(\p^1 \times \p^1)$ we may compute its preimage under $\pi_{12}$ in charts by direct substitution:
\begin{equation}
x = u_1 v_1 = 0,\quad \quad x = u_2 v_2 = 0.
\end{equation}
These local equations reveal that the total transform has three irreducible components, namely $F_1, F_2$ (given by $v_1=0, v_2=0$ respectively) and the \emph{proper transform} of the line $\{x=0\}$, which can be understood as the closure in $\operatorname{Bl}_{q_1q_2}(\p^1 \times \p^1)$ of the preimage $\pi_{12}^{-1}\left(\left\{ x=0 \right\} \backslash \left\{q_1, q_2 \right\}\right)$. 
For the remainder we will use the standard notation for the total and proper transforms of divisors under blowups. 
For the current example, if the divisor $\{ x =0 \}$ on $\p^1 \times \p^1$ is  denoted by $H_x$, then we use the same symbol for its total transform on $\operatorname{Bl}_{q_1q_2}(\p^1 \times \p^1)$. Then the proper transform is written as $H_x - F_1 - F_2$, since unions of curves correspond to sums of divisors and we have the total transform written as the sum of its three irreducible components as $H_x = (H_x - F_1- F_2) + (F_1) + (F_2)$.

We see from the above equations that the intersections of the proper transform with $F_1, F_2$ are given in coordinates by $(u_1,v_1)=(0,0)$, $(u_2, v_2)=(0,0)$ respectively. From this, our previous calculations show that the vector field diverges on the proper transform of $\{x=0\}$ and no solution curves from elsewhere on the bundle pass through it: it is an \emph{inaccessible divisor}.

We now proceed with our analysis elsewhere on the bundle, beginning with the $(x,Y)$-chart. After direct substitution the only indeterminacy we find here is given by $(x,Y) = (0, N/\alpha)$, which is just $q_2$ in this chart. We again see that the line $\left\{x=0\right\}$ away from $q_2$ is inaccessible, and we identify another inaccessible divisor given by the line $\{Y=0\}$. In the chart $(X,y)$, we have another singularity at 
\begin{equation}
q_3 : (X,y) = (0, -n / N),
\end{equation}
where we note that, importantly, the first component $X'$ is regular but $y'$ is indeterminate. We observe again that the vector field, or more precisely its second component $y'$, diverges on $\{X=0\}$ away from $q_3$. Blowing up $q_3$ and introducing charts
\begin{equation} 
(X,y + n / N) = (u_3 v_3, v_3) = (V_3 , U_3 V_3),
\end{equation}
we see by a calculation similar to those in the cases of $q_1, q_2$ that the vector field is then regular on the exceptional divisor $F_3$ except at the point $(u_3, v_3) = (0,0)$, which is its intersection with the proper transform of the line $\{X = 0\}$.

At this point, we have resolved all singularities of the system away from $(X,Y)=(0,0)$ on $\p^1 \times \p^1$. Indeed, in the $(X,Y)$ chart we have 
\begin{subequations}
\begin{align}
X' &= \frac{(c+1)N Y - (2c+N-c^2N) X Y + 4c(c+1)N X^2 - 2c(c+1) \alpha X^2 Y}{4c^2 Y}, \\
Y' &= \frac{ -(c+1)N Y - (c+1) n Y^2 + 2c (c+1)N X^2 - 2c(c+1)\alpha X^2 Y}{4c^2 X},
\end{align}
\end{subequations}
so we see a new singularity at
\begin{equation}
q_4 : (X, Y) = (0,0),
\end{equation}
at which both components $X',Y'$ are indeterminate. Blowing this up and introducing charts $(u_4,v_4), (U_4,V_4)$ according to
\begin{equation}
(X,Y) = (u_4 v_4, v_4) = (V_4 , U_4 V_4),
\end{equation}
we see that, unlike $q_1, q_2, q_3$, the singularity is not immediately resolved. In particular, the vector field in the second chart is given by 
\begin{subequations}
\begin{align}
U_4' &= \frac{ - 2(c+1)N U_4 -2c(c+1)N V_4 +(2c+N-c^2N)U_4 V_4 - (c+1)n U_4^2 V_4}{4c^2 V_4}, \\
V_4' &= \frac{ (c+1)N U_4 +4c(c+1)N V_4 -(2c+N-c^2N)U_4 V_4 - 2c(c+1) \alpha U_4^2 V_4}{4c^2 U_4},
\end{align}
\end{subequations}
and we see that on the exceptional divisor $F_4$ there is still a singularity
\begin{equation}
q_5 : (U_4, V_4) = (0,0),
\end{equation}
at which both $U_4', V_4'$ are indeterminate, so we must blow up this point too. To resolve the singularity at $q_4$, we require five blowups in total, of points defined in coordinates as follows:
\begin{equation} \label{p4cascadecharts}
\begin{aligned}
&q_4 : (X,Y) = (0,0), &&(X,Y) = (u_4 v_4, v_4) = (V_4 , U_4 V_4), \\
&q_5 : (U_4, V_4) = (0,0), &&(U_4, V_4) = (u_5 v_5, v_5) = (V_5 , U_5 V_5), \\
&q_6 : (u_5,v_5) = (-2c,0), &&(u_5 + 2c, v_5) = (u_6 v_6, v_6) = (V_6 , U_6 V_6), \\
&q_7 : (u_6,v_6) = (2c(c-1),0), &&(u_6 - 2c(c-1), v_6) = (u_7 v_7, v_7) = (V_7 , U_7 V_7),
\end{aligned}
\end{equation}
\begin{equation*}
\begin{aligned}
&q_8 : (u_7,v_7) = ( - \frac{2c}{N} \left( N+ c^2 N + 2c( 1+n + \alpha - N) \right),0 ), \\
&~~\quad \quad \left(u_7 + \frac{2c}{N} \left( N + c^2 N + 2c( 1+n + \alpha - N) \right), v_7\right) = (u_8 v_8, v_8) = (V_8 , U_8 V_8).
\end{aligned}
\end{equation*}
By direct calculation in the two charts $(u_8,v_8)$ and $(U_8,V_8)$ we find that the vector field on the exceptional divisor $F_8$, where $v_8 = 0$ respectively $V_8=0$, has no more indeterminacies. After making the change of variables and then substituting $v_8=0$ this is given by
\begin{subequations} \label{u8v8zero}
\begin{align}
u_8' &= \frac{\mathcal{U}(c,n,N,\alpha)+ N(3c^2 N-3N+10c) u_8}{4c^2},\\
v_8' &= -\frac{(c+1)N}{4c^2},
\end{align}
\end{subequations}
where $\mathcal{U}$ is a known polynomial in $c, n, N, \alpha$ which we omit for conciseness. In the other chart, we make the change of variables then substitute $V_8=0$ to obtain
\begin{subequations} \label{U8V8zero}
\begin{align}
U_8' &= \frac{ U_8 \left( \mathcal{V}(c,n,N,\alpha) -N (3c^2 N-3N+10c) U_8 \right) }{4c^2 N} ,\\
V_8' &= - \frac{(c+1)N}{4c^2 U_8},
\end{align}
\end{subequations}
where similarly $\mathcal{V}$ is a known polynomial in $c, n, N, \alpha$. From the above results, we deduce that on the part of $F_8$ where $U_8 \neq 0$ the vector field is regular, and there are no more indeterminacies left to resolve. 

After the eight blowups of the $\p^1 \times \p^1$-fiber over $c$, we obtain a rational surface $\X_c$. 
We denote the group of divisors $\Div(\X_c)$, whose elements are formal integer sums of closed irreducible codimension one subvarieties of $\X_c$, the quotient of which by the subgroup $\operatorname{P}(\X_c)$ of \emph{principal divisors} (i.e by the relation of \emph{linear equivalence}) is the divisor class group 
\begin{equation}
\Cl(\X_c) = \Div(\X_c) / \operatorname{P}(\X_c).
\end{equation}
As $\X_c$ is smooth, this is isomorphic to the \emph{Picard group} (or \emph{Picard lattice}) $\Pic(\X_c)$, whose elements are line bundles on $\X_c$ with group operation being tensor product. We can write this as 
\begin{equation}
\Pic(\X_c) \cong \Cl(\X_c) = \Z \h_x + \Z \h_y + \Z \F_1 + \cdots \Z \F_8,
\end{equation}
where $\h_x, \h_y$ are the classes of total transforms of curves on $\p^1 \times \p^1$ of constant $x$ (or $X$) and $y$ (or $Y$) respectively, and $\F_i = [ F_i]$ is the class of the exceptional divisor of the blowup of $q_i$ (or more precisely its total transform under any further blowups of points on $F_i$). 
The Picard group $\Pic(\X_c)$ is equipped with the symmetric bilinear \emph{intersection form} defined by 
\begin{equation}\label{eq:int-form}
\begin{gathered}
\mathcal{H}_{x}\cdot \mathcal{H}_{x} = \mathcal{H}_{y}\cdot \mathcal{H}_{y} = \mathcal{H}_{x}\cdot \mathcal{F}_{i} = 
\mathcal{H}_{y}\cdot \mathcal{F}_{j} = 0, \\
\mathcal{H}_{x}\cdot \mathcal{H}_{y} = 1,\qquad  \mathcal{F}_{i}\cdot \mathcal{F}_{j} = - \delta_{ij},
\end{gathered}
\end{equation}
on the generators, extended by symmetry and linearity.
We give a schematic representation of the configuration of points which were blown up and also the resulting surface in \autoref{fig:pointconfiglaguerre}, in which blue curves represent inaccessible divisors.
We also collect the locations in coordinates of the eight points in \autoref{pointlocationsLaguerre}.

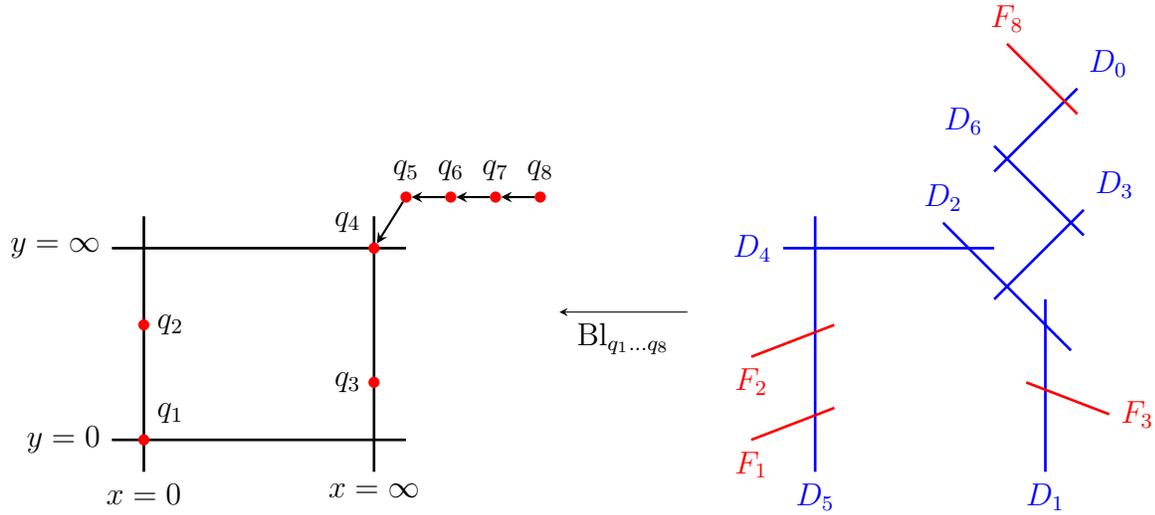
\begin{figure}[ht] 
	\begin{tikzpicture}[scale=.85,>=stealth,basept/.style={circle, draw=red!100, fill=red!100, thick, inner sep=0pt,minimum size=1.2mm}]
		\begin{scope}[xshift = -6cm]
			\draw [black, line width = 1pt] 	(4.1,2.5) 	-- (-0.5,2.5)	node [left]  {$y=\infty$} node[pos=0, right] {$$};
			\draw [black, line width = 1pt] 	(0,3) -- (0,-1)			node [below] {$x=0$}  node[pos=0, xshift=-7pt] {$$};
			\draw [black, line width = 1pt] 	(3.6,3) -- (3.6,-1)		node [below]  {$x=\infty$} node[pos=0, above, xshift=7pt] {$$};
			\draw [black, line width = 1pt] 	(4.1,-.5) 	-- (-0.5,-0.5)	node [left]  {$y=0$} node[pos=0, right] {$$};

			\node (p1) at (0,-0.5) [basept,label={[xshift=10pt, yshift = 0 pt] $q_{1}$}] {};
			\node (p2) at (0,1.3) [basept,label={[xshift=10pt, yshift = -10 pt] $q_{2}$}] {};
			\node (p3) at (3.6,0.4) [basept,label={[xshift=-10pt, yshift=-10pt] $q_{3}$}] {};
			\node (p4) at (3.6,2.5) [basept,label={[xshift=-10pt,yshift=0pt] $q_{4}$}] {};
			\node (p5) at (4.1,3.3) [basept,label={[xshift=0pt, yshift = 0 pt] $q_{5}$}] {};
			\node (p6) at (4.8,3.3) [basept,label={[xshift=0pt, yshift = 0 pt] $q_{6}$}] {};
			\node (p7) at (5.5,3.3) [basept,label={[xshift=0pt, yshift = 0 pt] $q_{7}$}] {};
			\node (p8) at (6.2,3.3) [basept,label={[xshift=0pt, yshift = 0 pt] $q_{8}$}] {};
			\draw [line width = 0.8pt, ->] (p5) -- (p4);
			\draw [line width = 0.8pt, ->] (p6) -- (p5);
			\draw [line width = 0.8pt, ->] (p7) -- (p6);
			\draw [line width = 0.8pt, ->] (p8) -- (p7);
		\end{scope}
	
		\draw [->] (2.5,1.5)--(0.5,1.5) node[pos=0.5, below] {$\text{Bl}_{q_1\dots q_8}$};
	
		\begin{scope}[xshift = 4.5cm]
			\draw [blue, line width = 1pt] 	(2.8,2.5) 	-- (-0.5,2.5)	node [left]  {$D_4$} node[pos=0, right] {};
			\draw [blue, line width = 1pt] 	(0,3) -- (0,-1)			node [below] {$D_5$}  node[pos=0, above, xshift=-7pt] {};
			\draw [blue, line width = 1pt] 	(3.6,1.7) -- (3.6,-1)		node [below]  {$D_1$} node[pos=0, above, xshift=7pt] {};

			\draw [red, line width = 1pt] 	(.3,0) 	-- (-1,-.5)	 node [left] {} node[ below] {$F_1$};
			
			\draw [red, line width = 1pt] 	(.3,1.3) 	-- (-1,.8)	 node [left] {} node[ below] {$F_2$};
			
			\draw [red, line width = 1pt] 	(3.3,0.4) 	-- (4.6,-.1)	 node [left] {} node[right] {$F_3$};
			
			\draw [blue, line width = 1pt]	(4,.9) -- (2,2.9)  node [ above ] {$D_2$};
			\draw [blue, line width = 1pt]	(2.8,1.7) -- (4.2,3.1) node[above right] {$D_3$};
			\draw [blue, line width = 1pt]	(4.2,2.7) -- (2.8,4.1) node[above left] {$D_6$};
			\draw [blue, line width = 1pt]	(2.8,3.7) -- (4.1,5)  node [above right] {$D_0$};
			\draw [red, line width = 1pt]	(4.1,4.6) -- (3,5.7) node [above] {$F_8$};

%

		\end{scope}
	\end{tikzpicture}
	\caption{Point configuration and surface for the semi-classical Laguerre weight}
	\label{fig:pointconfiglaguerre}

\end{figure}

\begin{figure}[ht]
\begin{equation*}
\begin{gathered}
q_1 : (x,y) = (0,0),  \quad q_2 : (x,y) = (0, \alpha/N),  \quad q_3 : (x, y) = (\infty,- n / N), \\
\\
\begin{aligned}
&q_4 : (x,y) = (\infty, \infty) ~ \leftarrow  &~   &q_5 : (U_4, V_4) = (x/y, 1/x) = (0,0) \\
&  & &  \uparrow  & \\
&  & & q_6 : (u_5,v_5) = (U_4/V_4, V_4) = (-2c, 0)  \\
&  & &  \uparrow  & \\
&  & & q_7 : (u_6,v_6) = \left( (u_5+2c)/v_5 , v_5 \right) = (2c(c-1), 0)  \\
&  & &  \uparrow  & \\
&  & & q_8 : (u_7,v_7) =\left( (u_6 -2c(c-1))/v_6 , v_6 \right) \\
&  & & \quad \quad ~~~~~~~~~ = ( - 2c\left( N+ c^2 N + 2c( \alpha - N + n +1 ) / N \right), 0) 
\end{aligned}
\end{gathered}
\end{equation*}

\caption{Point locations for the semi-classical Laguerre weight}
 \label{pointlocationsLaguerre}
 \end{figure}

The only step remaining in the construction of a space of initial conditions for the system \eqref{differential} is to identify and remove the inaccessible divisors from each fiber $\X_c$. 
In our earlier calculations, we identified the following inaccessible divisors: the proper transform of the line $\left\{x=0 \right\}$, which is the divisor $H_x - F_1- F_2 \in \Div(\X_c)$, the proper transform of the line $\left\{X=0\right\}$, which is $H_x - F_3 - F_4$, and the proper transform of the line $\left\{Y=0 \right\}$, which is $H_y - F_4- F_5$ (as $q_5$ was found at the intersection of $F_4$ and the proper transform of $\left\{Y=0 \right\}$ under the blowup of $q_4$). Further calculations in charts reveal that the vector field diverges everywhere in $\pi^{-1}(q_4)$ except for the part of $F_8$ where $U_8 \neq 0$, so we also have the following inaccessible divisors: $F_4 - F_5$, $F_5 - F_6$, $F_6- F_7$, $F_7 - F_8$. 
We denote these by $D_i$ as follows:
\begin{equation} \label{DiLaguerre}
\begin{aligned}
D_0  &= F_7-F_8, 		&\quad  	&D_4 = H_y - F_4 - F_5, \\
D_1 &= H_x - F_3 - F_4, 	&\quad 	&D_5 = H_x - F_1 - F_2, \\
D_2 &= F_4 - F_5, 		&\quad 	&D_6 = F_6 - F_7, \\
D_3 &= F_5 - F_6.
\end{aligned}
\end{equation}
We remove the union of these curves $D_{\operatorname{red}} = \bigcup_{i=0}^{6} D_i $ from the surface $\X_c$ to arrive at a bundle over $B = \C \backslash \{0\}$ with fiber over $c$ given by $\X_c \backslash D_{\operatorname{red}}$, which we denote
\begin{equation}
\begin{gathered}
\rho : E \longrightarrow B, \\
\rho^{-1}(c) = \X_c \backslash D_{\operatorname{red}}.
\end{gathered}
\end{equation}
This bundle admits a \emph{uniform foliation} \cite{Ok79, T1} by solution curves transverse to the fibres, and each fibre can be regarded as a space of initial conditions for the system \eqref{differential}.

\subsection{The surface type}
We now perform Step 2 of the identification procedure, showing that $\X_c$ constructed above is an example of the surfaces appearing in Sakai's theory, and determine its type in the classification. 
This is done by identifying an \emph{anticanonical divisor} of the surface $\X_c$, which is by definition the pole divisor of a rational 2-form on $\X_c$,
This 2-form will also provide the symplectic structure of the resulting Hamiltonian system on the bundle $E$ (see \autoref{sect333}).
As we are dealing with an eight-point blowup of $\p^1 \times \p^1$ the class in $\Pic(\X_c)$ of the anticanonical divisor should be 
\begin{equation}
- \mathcal{K}_{\X_c} = 2 \h_x + 2 \h_y - \F_1 - \F_2 - \F_3 - \F_4 - \F_5 - \F_6 - \F_7 - \F_8.  
\end{equation} 
This divisor can be identified from the configuration of the points $q_1,\dots, q_8$ by finding a biquadratic curve on $\p^1 \times \p^1$ on which these points lie.
In this case we can take the curve as the pole divisor of the rational 2-form
\begin{equation} \label{omegaLaguerre}
\omega = k \frac{dx \wedge dy}{x} = - k \frac{dX \wedge dy}{X} = - k \frac{dx \wedge dY}{Y^2} = k \frac{dX \wedge dY}{X Y^2},
\end{equation}
where $k$ is a nonzero constant which we allow to be arbitrary at this stage.
The pole divisor of $\omega$ passes through $q_1,\dots, q_8$, including the infinitely near points in the cascade over $q_4$. 
The proper transform of the biquadratic curve in $\p^1\times\p^1$ under the blowups (or equivalently the pole divisor of the 2-form $\omega$ lifted to $\X_c$) gives an anticanonical divisor $D \in | -\mathcal{K}_{\X_c}|$, whose irreducible components are exactly the inaccessible divisors identified when we resolved the indeterminacies of the system above. 
This is given in terms of $D_i$ introduced in \eqref{DiLaguerre} as follows:
\begin{equation} \label{ACdivP4}
D = D_0 + D_1 +2 D_2 + 3 D_3 + 2 D_4+D_5 + 2 D_6.
\end{equation}

Sakai defined a \emph{generalized Halphen surface} to be a complex nonsingular projective surface with an anticanonical divisor of \emph{canonical type}. That is, a surface $\X$ with an anticanonical divisor $D \in | -\mathcal{K}_{\X}|$ whose decomposition $D= \sum_i m_i D_i$ into irreducible components is such that $[D_i] \cdot \mathcal{K}_{\X}=0$ for all $i$, where $[D_i]$ is the class of $D_i$ in $\Pic(\X)$.
A generalized Halphen surface has either $\dim | -\mathcal{K}_{\X}| = 1$, in which case we have a pencil of anticanonical divisors and $\X$ is a rational elliptic surface, or $\dim | -\mathcal{K}_{\X}| = 0$, in which case there is a \emph{unique} anticanonical divisor and we call $\X$ a \emph{Sakai surface}, which is the kind associated with Painlev\'e equations.
 
Sakai surfaces are classified firstly according to an affine Dynkin diagram determined by the matrix $(A_{ij})$ giving the intersection configuration of the components of the anticanonical divisor $[D_i] \cdot [D_j] = A_{ij}$.
The type $R$ of the affine Dynkin diagram associated with this intersection configuration is the \emph{surface type}, and the classes $\delta_i = [D_i]$ in $\Pic(\X_c)$ of the irreducible components of the anticanonical divisor form a basis of simple roots for an affine root system, which we call the \emph{surface root basis}, and whose $\Z$-span in $\Pic(\X_c)$ is the \emph{surface root lattice}, denoted 
\begin{equation}
Q(R) = Q(E_6^{(1)}) = \operatorname{span}_{\Z} \left\{ \delta_0, \delta_1, \delta_2, \delta_3, \delta_4, \delta_5, \delta_6 \right\} \subset \Pic(\X_c). 
\end{equation}
\begin{proposition}
The surface $\X_c$ constructed above is a Sakai surface, with surface type $E_6^{(1)}$.
\end{proposition}
\begin{proof}
It can be verified by direct calculation that $D$ given by \eqref{ACdivP4} is an anticanonical divisor;  
 with the components as in \autoref{fig:d-roots-hw}, we see that 
\begin{equation}
\begin{aligned}
[D] &= \delta_0 + \delta_1 +2 \delta_2 + 3 \delta_3 + 2 \delta_4+ \delta_5 + 2 \delta_6 \\
&= 2 \h_x + 2 \h_y - \F_1 - \F_2 - \F_3 - \F_4 - \F_5 - \F_6 - \F_7 - \F_8,
\end{aligned}
\end{equation}
so $D \in | -\mathcal{K}_{\X_c} |$. 
The fact that $\dim | -\mathcal{K}_{\X}| = 0$ can be verified by noting that the pole divisor of the symplectic form \eqref{omegaP4} is the unique bi-quadratic curve in $\p^1 \times \p^1$ passing through all eight points $q_1,\dots, q_8$. 
The fact that this anticanonical divisor is of canonical type can be verified by direct calculation of $\delta_i \cdot \mathcal{K}_{\X_c}$ using the formulae \eqref{eq:int-form}.
Finally, we compute the intersection configuration of the classes $\delta_i$ to be given by
\begin{equation}
- (\delta_i \cdot \delta_j) = \left(
\begin{array}{ccccccc}
 2 & 0 & 0 & 0 & 0 & 0 & -1 \\
 0 & 2 & -1 & 0 & 0 & 0 & 0 \\
 0 & -1 & 2 & -1 & 0 & 0 & 0 \\
 0 & 0 & -1 & 2 & -1 & 0 & -1 \\
 0 & 0 & 0 & -1 & 2 & -1 & 0 \\
 0 & 0 & 0 & 0 & -1 & 2 & 0 \\
 -1 & 0 & 0 & -1 & 0 & 0 & 2 \\
\end{array}
\right),
\end{equation}
which is the generalized Cartan matrix of affine type $E_6^{(1)}$ \cite{KAC}, whose Dynkin diagram we give in \autoref{fig:d-roots-hw}.
\end{proof}

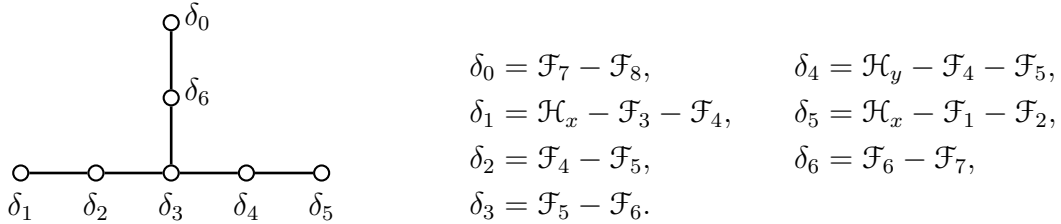
\begin{figure}[H]
\begin{equation*}\label{eq:d-roots-hw}
	\raisebox{-32.1pt}{\begin{tikzpicture}[
			elt/.style={circle,draw=black!100,thick, inner sep=0pt,minimum size=2mm}]
		\path 	(-2,0) 	node 	(d1) [elt, label={[xshift=-0pt, yshift = -25 pt] $\delta_{1}$} ] {}
		        (-1,0) node 	(d2) [elt, label={[xshift=-0pt, yshift = -25 pt] $\delta_{2}$} ] {}
		        ( 0,0) 	node  	(d3) [elt, label={[xshift=0pt, yshift = -25 pt] $\delta_{3}$} ] {}
		        ( 1,0) 	node  	(d4) [elt, label={[xshift=0pt, yshift = -25 pt] $\delta_{4}$} ] {}
		        ( 2,0) node 	(d5) [elt, label={[xshift=0pt, yshift = -25 pt] $\delta_{5}$} ] {}
		        ( 0,1) node 	(d6) [elt, label={[xshift=10pt, yshift = -10 pt] $\delta_{6}$} ] {}
		        ( 0,2) node 	(d0) [elt, label={[xshift=10pt, yshift = -10 pt] $\delta_{0}$} ] {};
		\draw [black,line width=1pt ] (d1) -- (d2) -- (d3) -- (d4)  -- (d5);
		\draw [black,line width=1pt ] (d3) -- (d6) -- (d0) ;
	\end{tikzpicture}} \qquad \qquad
			\begin{alignedat}{2}
			\delta_{0} &= \F_7-\F_8, &\qquad  \delta_{4} &= \h_y - \mathcal{F}_{4} - \mathcal{F}_{5},\\
			\delta_{1} &= \h_x - \F_3-\F_4, &\qquad  \delta_{5} &= \mathcal{H}_{x} -\F_1 - \F_2,\\
			\delta_{2} &= \F_4 - \F_5, &\qquad  \delta_{6} &= \F_6-\F_7,\\
			\delta_{3} &= \F_5 - \F_6.
			\end{alignedat}
\end{equation*}
	\caption{The Surface Root Basis for the semi-classical Laguerre weight}
	\label{fig:d-roots-hw}
\end{figure}

\subsection{Standard model of $E_6^{(1)}$-surfaces and $\pain{IV}$}

At this stage we have determined the surface type for the differential system \eqref{differential}, which tells us which Painlev\'e equation it should be transformable to, namely the fourth Painlev\'e equation $\pain{IV}$.  
For this we require the standard model of Sakai surfaces of type $E_6^{(1)}$ as provided by \cite{KajNouYam:2017:GAOPE}, with which we aim to identify those constructed in the previous section. 
These surfaces provide the space of initial conditions for the fourth Painlev\'e equation in the Hamiltonian form
\begin{equation} \label{hamP4}
\left\{
\begin{gathered}
\frac{dq}{dt} = \frac{\partial H}{\partial p} = q( 2p - q - t) - a_1, \\
 \frac{dp}{dt} = -\frac{\partial H}{\partial q} = p( 2q - p + t) + a_2 ,
\end{gathered}
\right.
\quad 
H(q,p,t) = q p (p - q - t)- a_1 p - a_2 q,
\end{equation}
where $a_1, a_2$ are free complex parameters. This system is equivalent to $\pain{IV}$ for $q(t)$ in the usual scalar form
\begin{equation}
q'' = \frac{1}{2q} (q')^2 + \frac{3}{2} q^3 + 2 t q^2 +\left( \alpha + \frac{t^2}{2} \right) + \frac{\beta}{2 q}, \quad \text{where} \quad \alpha = a_2 - a_0, ~ \beta = - a_1^2.
\end{equation}
Here $a_0$ is an extra parameter which will also appear in the point configuration, which is determined by the normalisation 
\begin{equation} \label{normalisationP4}
a_0 + a_1 + a_2 = 1.
\end{equation}

\subsubsection{Point configuration and anticanonical divisor}

To construct the surfaces, we begin with affine coordinates $(q,p)$, from which we introduce $Q=1/q, P=1/p$ similarly to in \autoref{section31}, so we have $\p^1 \times\p^1$  covered by the four affine charts $(q,p)$, $(Q,p)$, $(q,P)$ and $(Q,P)$. 
There are three points in $\p^1 \times \p^1$ initially identifiable as requiring blowups. For two of these, we require two blowups to resolve the singularities of the differential system, while the third requires four successive blowups. We show the configuration of these points in \autoref{fig:pointconfigE61}, with their locations in coordinates in \autoref{pointlocationsP4}. 
We use the same convention for introducing blowup coordinates as above, with two charts $(\tilde{u}_i , \tilde{v}_i), (\tilde{U}_i , \tilde{V}_i)$ introduced after the blowup of $p_i$, in which the exceptional divisor is given by $\tilde{v}_i = 0$, $\tilde{V}_i = 0$ respectively. 

We denote the surface obtained through these eight blowups by $\tilde{\X}_t$, with exceptional divisors arising from the eight blowups $E_1 ,\dots E_8 \in \Div(\tilde{\X}_t)$. 
The inaccessible divisors, which we denote by $\tilde{D}_i$, are given by 
\begin{equation} \label{DiP4}
\begin{aligned}
\tilde{D}_0  &= E_7-E_8, 			&\quad  		&\tilde{D}_4 = H_p - E_3 - E_5, \\
\tilde{D}_1 &= E_1 - E_2, 			&\quad 		&\tilde{D}_5 = E_3 - E_4, \\
\tilde{D}_2 &= H_q - E_1 - E_5, 	&\quad 		&\tilde{D}_6 = E_6 - E_7, \\
\tilde{D}_3 &= E_5 - E_6.
\end{aligned}
\end{equation}
We write the Picard lattice of the surface $\tilde{\X}_t$ in terms of generators as 
\begin{equation}
\Pic(\tilde{\X}_t) \cong \Z \h_q + \Z \h_p + \Z \E_1 + \cdots + \Z \E_8,
\end{equation}
where $\h_q, \h_p$ are the classes of total transforms of curves on $\p^1 \times \p^1$ of constant $q$ (or $Q$) and $p$ (or $P$) respectively, and $\E_i = [ E_i]$ are the exceptional classes arising from the blowups.
Again the inaccessible divisors give a representative of the anticanonical divisor class, given in terms of the surface roots $\tilde{\delta}_i = [\tilde{D}_i]$ as in \autoref{fig:d-roots-P4} by 
\begin{equation}
\begin{aligned}
-\mathcal{K}_{\tilde{\X}_t} &= \tilde{\delta}_0 + \tilde{\delta}_1 +2 \tilde{\delta}_2 + 3 \tilde{\delta}_3 + 2 \tilde{\delta}_4+ \tilde{\delta}_5 + 2 \tilde{\delta}_6 \\
&= 2 \h_q + 2 \h_p - \E_1 - \E_2 - \E_3 - \E_4 - \E_5 - \E_6 - \E_7 - \E_8.
\end{aligned}
\end{equation}

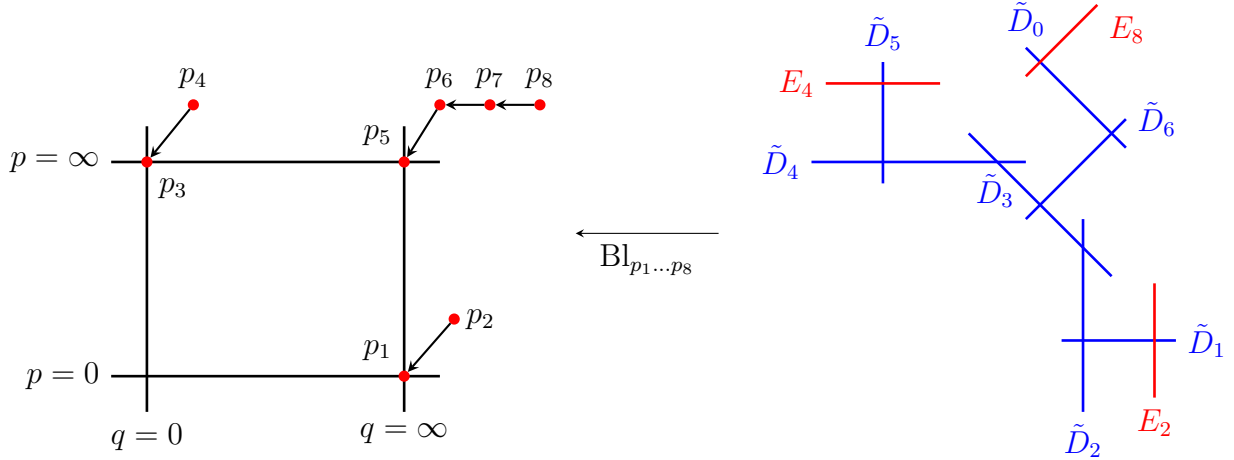
\begin{figure}[H]
	\begin{tikzpicture}[scale=.95,>=stealth,basept/.style={circle, draw=red!100, fill=red!100, thick, inner sep=0pt,minimum size=1.2mm}]
		\begin{scope}[xshift = -4cm]
			\draw [black, line width = 1pt] 	(4.1,2.5) 	-- (-0.5,2.5)	node [left]  {$p=\infty$} node[pos=0, right] {$$};
			\draw [black, line width = 1pt] 	(0,3) -- (0,-1)			node [below] {$q=0$}  node[pos=0, above, xshift=-7pt] {} ;
			\draw [black, line width = 1pt] 	(3.6,3) -- (3.6,-1)		node [below]  {$q=\infty$} node[pos=0, above, xshift=7pt] {};
			\draw [black, line width = 1pt] 	(4.1,-.5) 	-- (-0.5,-0.5)	node [left]  {$p=0$} node[pos=0, right] {$$};

			\node (p1) at (3.6,-.5) [basept,label={[xshift=-10pt, yshift = 0 pt] $p_{1}$}] {};
			\node (p2) at (4.3,0.3) [basept,label={[xshift=10pt, yshift = -10 pt] $p_{2}$}] {};
			\node (p3) at (0,2.5) [basept,label={[yshift=-20pt, xshift=+10pt] $p_{3}$}] {};
			\node (p4) at (0.65,3.3) [basept,label={[xshift=0pt, yshift = 0 pt] $p_{4}$}] {};
			\node (p5) at (3.6,2.5) [basept,label={[xshift=-10pt,yshift=0pt] $p_{5}$}] {};
			\node (p6) at (4.1,3.3) [basept,label={[xshift=0pt, yshift = 0 pt] $p_{6}$}] {};
			\node (p7) at (4.8,3.3) [basept,label={[xshift=0pt, yshift = 0 pt] $p_{7}$}] {};
			\node (p8) at (5.5,3.3) [basept,label={[xshift=0pt, yshift = 0 pt] $p_{8}$}] {};
			\draw [line width = 0.8pt, ->] (p2) -- (p1);
			\draw [line width = 0.8pt, ->] (p4) -- (p3);
			\draw [line width = 0.8pt, ->] (p6) -- (p5);
			\draw [line width = 0.8pt, ->] (p7) -- (p6);
			\draw [line width = 0.8pt, ->] (p8) -- (p7);
		\end{scope}
	
		\draw [->] (4,1.5)--(2,1.5) node[pos=0.5, below] {$\text{Bl}_{p_1\dots p_8}$};
	
		\begin{scope}[xshift = 5.5cm]
			\draw [blue, line width = 1pt] 	(2.8,2.5) 	-- (-0.2,2.5)	node [left]  {$\tilde{D}_4$} node[pos=0, right] {};
			\draw [blue, line width = 1pt] 	(3.6,1.7) -- (3.6,-1)		node [below]  {$\tilde{D}_2$} node[pos=0, above, xshift=7pt] {};

			\draw [blue, line width = 1pt] 	(.8,2.2) 	-- (.8,3.9)	 node [left] {} node[above] {$\tilde{D}_5$};
			\draw [red, line width = 1pt] 	(0,3.6) 	-- (1.6,3.6)	 node [left] {} node[pos=0, left] {$E_4$};
			
			\draw [blue, line width = 1pt] 	(3.3,0) node[left]{}	-- (4.9,0)	 node [left] {} node[right] {$\tilde{D}_1$};
			\draw [red, line width = 1pt] 	(4.6,-0.8) 	-- (4.6,.8)	 node [left] {} node[pos=0, below] {$E_2$};
			
			\draw [blue, line width = 1pt]	(4,.9) -- (2,2.9) node[ pos=.6, left] {$\tilde{D}_3$};
			\draw [blue, line width = 1pt]	(2.8,1.7) -- (4.2,3.1) node [right] {$\tilde{D}_6$} ;
			\draw [blue, line width = 1pt]	(4.2,2.7) -- (2.8,4.1) node [above] {$\tilde{D}_0$};
			\draw [red, line width = 1pt]	(2.8,3.7) -- (3.8,4.7)  node [below right] {$E_8$};

%

		\end{scope}
	\end{tikzpicture}
	\caption{Point configuration for the standard model of $E_6^{(1)}$-surfaces}
	\label{fig:pointconfigE61}
\end{figure}

\begin{figure}[H]
\begin{equation*} 
\begin{aligned}
&p_1 : (q,p) = (\infty,0) &~ 	\leftarrow &~   &~  &p_2 : (\tilde{u}_1, \tilde{v}_1) = ( q p, 1/q) = (-a_2, 0) \\
&p_3 : (q, p) = (0,\infty) &~ 	\leftarrow &~   &~  &p_4 : (\tilde{u}_3, \tilde{v}_3) = ( q p , 1/p) = (a_1, 0)    \\
&p_5 : (q,p) = (\infty,\infty) &~ 	\leftarrow &~   &~  &p_6 : (\tilde{u}_5, \tilde{v}_5) =  (p/q, 1/p) = (1,0)  \\
& & & & & \uparrow \\
& & & & &p_7 : (\tilde{u}_6, \tilde{v}_6) = \left( (\tilde{u}_5 - 1)/\tilde{v}_5, \tilde{v}_5\right) = (t,0) \\
& & & & & \uparrow \\
& & & & &p_8 : (\tilde{u}_7, \tilde{v}_7 ) = \left( (\tilde{u}_6 - t ) / \tilde{v}_6 , \tilde{v}_6 \right) = (t^2 - a_0,0) 
\end{aligned}
\end{equation*}
\caption{Point locations for the standard model of $E_6^{(1)}$-surfaces}
\label{pointlocationsP4}
\end{figure}


\begin{figure}[H]
\begin{equation*}\label{eq:d-roots-hw}
	\raisebox{-32.1pt}{\begin{tikzpicture}[
			elt/.style={circle,draw=black!100,thick, inner sep=0pt,minimum size=2mm}]
		\path 	(-2,0) 	node 	(d1) [elt, label={[xshift=-0pt, yshift = -25 pt] $\tilde{\delta}_{1}$} ] {}
		        (-1,0) node 	(d2) [elt, label={[xshift=-0pt, yshift = -25 pt] $\tilde{\delta}_{2}$} ] {}
		        ( 0,0) 	node  	(d3) [elt, label={[xshift=0pt, yshift = -25 pt] $\tilde{\delta}_{3}$} ] {}
		        ( 1,0) 	node  	(d4) [elt, label={[xshift=0pt, yshift = -25 pt] $\tilde{\delta}_{4}$} ] {}
		        ( 2,0) node 	(d5) [elt, label={[xshift=0pt, yshift = -25 pt] $\tilde{\delta}_{5}$} ] {}
		        ( 0,1) node 	(d6) [elt, label={[xshift=10pt, yshift = -10 pt] $\tilde{\delta}_{6}$} ] {}
		        ( 0,2) node 	(d0) [elt, label={[xshift=10pt, yshift = -10 pt] $\tilde{\delta}_{0}$} ] {};
		\draw [black,line width=1pt ] (d1) -- (d2) -- (d3) -- (d4)  -- (d5);
		\draw [black,line width=1pt ] (d3) -- (d6) -- (d0) ;
	\end{tikzpicture}} \qquad \qquad
			\begin{alignedat}{2}
			\tilde{\delta}_{0} &= \E_7-\E_8, &\qquad  \tilde{\delta}_{4} &= \h_p - \mathcal{E}_{3} - \mathcal{E}_{5},\\
			\tilde{\delta}_{1} &= \E_1-\E_2, &\qquad  \tilde{\delta}_{5} &= \E_3 - \E_4,\\
			\tilde{\delta}_{2} &= \h_q - \E_1 - \E_5, &\qquad  \tilde{\delta}_{6} &= \E_6-\E_7,\\
			\tilde{\delta}_{3} &= \E_5 - \E_6.
			\end{alignedat}
\end{equation*}
	\caption{The Surface Root Basis for the standard model of $E_6^{(1)}$-surfaces}
	\label{fig:d-roots-P4}
\end{figure}
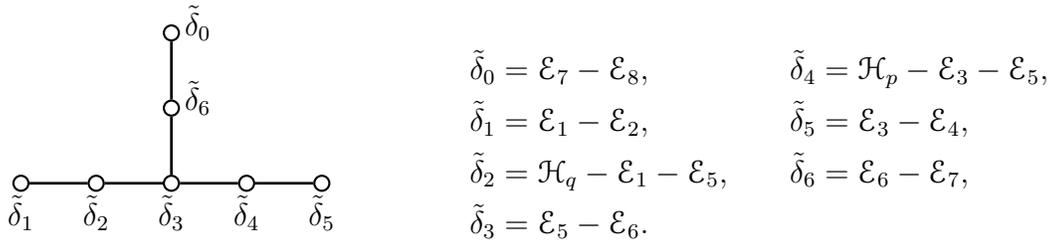

\subsubsection{Period mapping and root variables} \label{sec332}

The locations of the eight points which were blown up to obtain the surface $\tilde{\X}_t$ depend on, in addition to the independent variable $t$, three parameters $a_0, a_1, a_2$ subject to the normalisation \eqref{normalisationP4}. 
These parameters are the \emph{root variables} for the surface $\tilde{\X}_t$, which correspond to a choice of root basis for another affine root lattice in $\Pic(\tilde{\X}_c)$, which we call the \emph{symmetry root lattice}. 
This is given by the orthogonal complement in $\Pic(\tilde{\X}_c)$ (with respect to the intersection form) of the surface root lattice $Q(R)$, and is denoted by $Q(R^{\perp})$, where the \emph{symmetry type} of the surface is the type $R^{\perp}$ of its affine Dynkin diagram.
With the choice of \emph{symmetry root basis} as in \cite{KajNouYam:2017:GAOPE} we have 
\begin{equation}
Q(R^{\perp}) = Q(A_2^{(1)}) = \operatorname{span}_{\Z} \left\{ \tilde{\alpha}_0, \tilde{\alpha}_1, \tilde{\alpha}_2 \right\} \subset \Pic(\tilde{\X}_t),
\end{equation} 
where the symmetry roots $\tilde{\alpha}_0, \tilde{\alpha}_1, \tilde{\alpha}_2$ are given in \autoref{fig:symmetryrootsP4}.

\begin{figure}[H]
\begin{equation*}\label{eq:d-roots-hw}
	\raisebox{-32.1pt}{\begin{tikzpicture}[
			elt/.style={circle,draw=black!100,thick, inner sep=0pt,minimum size=2mm}]
		\path 	( -1,0) 	node  	(a1) [elt, label={[xshift=0pt, yshift = -25 pt] $\tilde{\alpha}_{1}$} ] {}
		        ( 1,0) 	node  	(a2) [elt, label={[xshift=5pt, yshift = -25 pt] $\tilde{\alpha}_{2}$} ] {}
		        ( 0,1.3) node 	(a0) [elt, label={[xshift=10pt, yshift = -5 pt] $\tilde{\alpha}_{0}$} ] {};
		\draw [black,line width=1pt ] (a0) -- (a1)  -- (a2) -- (a0);
	\end{tikzpicture}} \qquad \qquad
			\begin{alignedat}{2}
			\tilde{\alpha}_{0} &= \h_q + \h_p - \E_5-\E_6 - \E_7 - \E_8, \\
			\tilde{\alpha}_{1} &= \h_q - \E_3-\E_4, \\
			\tilde{\alpha}_{2} &= \h_p - \E_1-\E_2,
			\end{alignedat}
\end{equation*}
	\caption{The Symmetry Root Basis for the standard model of $E_6^{(1)}$-surfaces}
	\label{fig:symmetryrootsP4}
\end{figure}

In order to define the period mapping, we choose a rational symplectic form on $\tilde{\X}_t$ whose pole divisor is the configuration of curves $\tilde{D}_i$ as in \eqref{DiP4}. Any such symplectic form is given in the affine $(q,p)$-chart by $\tilde{k} dq \wedge dp$, where $\tilde{k}$ is a nonzero constant which will be normalised later.
In order to calculate the value of the period mapping on the symmetry root basis, we will need to work with the symplectic form $\tilde{\omega}$ in a number of charts:
\begin{equation} \label{omegaP4}
\begin{aligned}
\tilde{\omega} &= \tilde{k} dq \wedge dp = - \tilde{k} \frac{dQ \wedge dp}{Q^2} = - \tilde{k} \frac{dq \wedge dP}{P^2} = \tilde{k} \frac{dQ \wedge dP}{Q^2 P^2} \\
&= \tilde{k} \frac{d \tilde{u}_1 \wedge d \tilde{v}_1}{\tilde{v}_1} = -\tilde{k} \frac{d \tilde{u}_3 \wedge d \tilde{v}_3}{\tilde{v}_3} = \tilde{k} \frac{d\tilde{u}_5 \wedge d \tilde{v}_5}{\tilde{u}_5^2 \tilde{v}_5^{3}}\\
&= \tilde{k} \frac{d \tilde{u}_6 \wedge \tilde{v}_6}{\tilde{v}_6^2(1 + \tilde{u}_6\tilde{v}_6)^2} = \tilde{k} \frac{d \tilde{u}_7 \wedge d \tilde{v}_7}{\tilde{v}_7( 1 + \tilde{v}_7 ( t + \tilde{u}_7 \tilde{v}_7 ))^2}.
\end{aligned}
\end{equation}
Choosing $\tilde{\omega}$ as above, the \emph{period mapping} is a linear map
\begin{equation}
\tilde{\chi} : Q(R^{\perp}) \longrightarrow \C, 
\end{equation}
which gives the \emph{root variables} corresponding to the symmetry root basis
\begin{equation}
a_i = \tilde{\chi}(\tilde{\alpha}_i).
\end{equation}
These are computed according to the following process for each $\tilde{\alpha}_i$ in the symmetry root basis (see \cite{Sak:2001:RSAWARSGPE} for details, or \cite{DT18, DFS19} for more examples of explicit calculations):
\begin{itemize}
\item First, express $\tilde{\alpha}_i$ as a difference of two effective divisors, $\tilde{\alpha}_i = [C_i^{1}] - [C_i^{0}]$, where $C_i^1$ and $C_i^0$ are irreducible curves;
\item Note the unique component $\tilde{D}_k$ of the anticanonical divisor such that $[\tilde{D}_k] \cdot [C_i^{1}] = [\tilde{D}_k] \cdot [C_i^{0}] = 1$. Denote the points where $\tilde{D}_k$ intersects the curves by $P_i = \tilde{D}_k \cap C_i^{0}$ and $Q_i = \tilde{D}_k \cap C_i^{1}$;
\item The period mapping can then be computed as
\begin{equation}
\tilde{\chi}(\tilde{\alpha}_i) = \tilde{\chi}([C_i^{1}] - [C_i^{0}]) = \int_{P_i}^{Q_i} \frac{1}{2 \pi \mathfrak{i}} \oint_{\tilde{D}_k} \tilde{\omega} = \int_{P_i}^{Q_i} \operatorname{res}_{\tilde{D}_k} \tilde{\omega}.
\end{equation}
\end{itemize}
For the surface $\tilde{\X}_t$ and the choice of symmetry root basis in \autoref{fig:symmetryrootsP4}, the root variables are provided by the following Lemma, which is proven by standard calculations using the method above.

\begin{lemma}

\begin{enumerate}[(a)]
\item The residue of the rational symplectic form $\tilde{\omega}$ given in charts by \eqref{omegaP4} along each of the irreducible components of the anticanonical divisor is given by
\begin{equation}
\begin{aligned}
\operatorname{res}_{\tilde{D}_0} \tilde{\omega} &= - \tilde{k} d \tilde{u}_7,  &\quad &\operatorname{res}_{\tilde{D}_1} \tilde{\omega} = -\tilde{k} d\tilde{u}_1, &\qquad &\operatorname{res}_{\tilde{D}_2} \tilde{\omega} = 0, \qquad \operatorname{res}_{\tilde{D}_3} \tilde{\omega} = 0,  \\
\operatorname{res}_{\tilde{D}_4} \tilde{\omega} &= 0,  &\quad &\operatorname{res}_{\tilde{D}_5} \tilde{\omega} = \tilde{k} d\tilde{u}_3,  &\quad &\operatorname{res}_{\tilde{D}_6} \tilde{\omega} =  0. 
\end{aligned}
\end{equation}
\item The values of the period mapping on the symmetry roots in \autoref{fig:symmetryrootsP4} are given by
\begin{equation}
\tilde{\chi}(\tilde{\alpha}_0) = - \tilde{k} a_0, \qquad  \tilde{\chi}(\tilde{\alpha}_1) = - \tilde{k} a_1, \qquad \tilde{\chi}(\tilde{\alpha}_2) = - \tilde{k} a_2.
\end{equation}
\item By normalising the symplectic form with $\tilde{k}=-1$, we have the parameters $a_i$ in the point configuration being the root variables for the surface $\tilde{\X}_t$ with the choice of symmetry root basis in \autoref{fig:symmetryrootsP4}.
This also yields the parameter normalisation 
\begin{equation}
a_0 + a_1 +a_2 = \tilde{\chi}(\tilde{\alpha}_0 + \tilde{\alpha}_1 + \tilde{\alpha}_2) = \tilde{\chi}(-\mathcal{K}_{\tilde{\X}_t}) = 1.
\end{equation}
\end{enumerate}

\end{lemma}

\subsubsection{Symplectic atlas and global Hamiltonian structure of $\Pain{IV}$ on Okamoto's space} \label{sect333}

In \cite{T1, T2, M, orbifold}, certain uniqueness results for Hamiltonian systems on the spaces constructed by Okamoto were proved, which give the theoretical basis for our method. These ensure that appropriate identifications between surfaces provide transformations relating differential systems to the standard Hamiltonian forms of Painlev\'e equations.
The uniqueness results relate to global Hamiltonian structures of the Painlev\'e equations on Okamoto's spaces, which we recall now for the present example.

With the normalisation obtained in the previous section, the symplectic form $\tilde{\omega} = dp \wedge dq$ on the fibre and Hamiltonian $H(q,p,t)$ from \eqref{hamP4} allow us to define a global Hamiltonian structure of $\Pain{IV}$ on the bundle forming Okamoto's space as follows. 
Similarly to in \autoref{section31} we denote this bundle by
\begin{equation}
\begin{gathered}
\tilde{\rho} : \tilde{E} \longrightarrow \tilde{B}, \\
\tilde{\rho}^{-1}(t) = \tilde{E}_t = \tilde{\X}_t \backslash \tilde{D}_{\operatorname{red}},
\end{gathered}
\end{equation}
where $\tilde{B} = \C$ is the independent variable space for $\pain{IV}$ and $\tilde{D}_{\operatorname{red}} = \bigcup_{i=0}^6 \tilde{D}_i$ is the union of the inaccessible divisors.
We take an atlas for the total space $E$, coming from the affine $(q,p)$-chart for the surface $\tilde{\X}_t$ as well as charts to cover affine parts of the exceptional divisors $E_2, E_4, E_8$ which are not contained in $\tilde{D}_{\operatorname{red}}$.
For our purposes we can require this atlas to be \emph{symplectic}, but in a slightly weaker sense than that of \cite{T1} since we do not require the symplectic form on the fibre to be written in canonical coordinates, but just to be \emph{independent of} $t$:
\begin{definition}
A symplectic atlas for the bundle $\tilde{E}$ is one for which the transition functions are birational and the symplectic form in each chart is independent of the point in the base space. The transition functions are of the form 
\begin{equation}
\begin{gathered}
\varphi : \C^3 \ni (x,y,t) \mapsto (X,Y,t) \in \C^3 ,\\
\varphi_t : (x,y) \mapsto (X,Y), \\
(\varphi_t)_* (F(x,y) \delta x \wedge \delta y ) = G(X,Y) \delta X \wedge \delta Y,
\end{gathered}
\end{equation}
where again $\delta$ is the exterior derivative on the fibre over $t$, and $F, G$ are rational functions independent of $t$.
\end{definition}
With such an atlas, the following Lemma allows for a differential system which is Hamiltonian in one chart to extend to the whole bundle with a Hamiltonian structure.
\begin{lemma} \label{symplecticlemma}
If we have a symplectic atlas for a bundle $\tilde{E}$ as above, then given $H(x,y,t)$ there exists $K(X,Y,t)$ (unique modulo functions of $t$) such that 
\begin{equation}
\varphi_* \left( F(x,y) dy \wedge dx - dH \wedge dt \right) = G(X,Y) dY \wedge dX - dK \wedge dt.
\end{equation}
Further, the Hamiltonian system 
\begin{equation}
F(x,y) \frac{dx}{dt} = \frac{\partial H}{\partial y}, \qquad F(x,y) \frac{dy}{dt} = - \frac{\partial H}{\partial x}
\end{equation} 
is transformed under $\varphi$ to 
\begin{equation}
G(X,Y) \frac{dX}{dt} = \frac{\partial K}{\partial Y}, \qquad G(X,Y) \frac{dY}{dt} = - \frac{\partial K}{\partial X}.
\end{equation} 
\end{lemma}
\begin{proof}
This is a slight generalisation of an elementary fact about symplectic transformations as presented in \cite{T1}, and is proved by direct calculation.
\end{proof}
For the bundle in this case, we can take the atlas to be
\begin{equation}
\tilde{E} = \C^3_{q,p,t} \cup \C^3_{\tilde{u}_2,\tilde{v}_2,t} \cup \C^3_{\tilde{u}_4,\tilde{v}_4,t} \cup \C^3_{\tilde{r}_8,\tilde{s}_8,t},
\end{equation}
with gluing defined by 
\begin{equation}
\begin{aligned}
\frac{1}{q} &= \tilde{v}_2( -a_2^{-1} + \tilde{u}_2 \tilde{v}_2) , &&\quad p = \tilde{v}_2, \\
q &= \tilde{v}_4( a_1 + \tilde{u}_4 \tilde{v}_4), &&\quad \frac{1}{p} = \tilde{v}_4, \\
\frac{1}{q} &= \tilde{s}_8, &&\quad p= \frac{1 +  t \tilde{s}_8 -   a_0 \tilde{s}_8^{2} - \tilde{r}_8 \tilde{s}_8^{3} }{\tilde{s}_8}.
\end{aligned}
\end{equation}
The total space $E$ now has the structure of a fibre bundle of holomorphic symplectic manifolds, with fibre $E_t = \X_t \backslash D_{\operatorname{red}}$ over $ t \in \C$ equipped the following symplectic form:
\begin{equation}
\tilde{\omega}_t = \delta p \wedge \delta q = \frac{\delta u_2 \wedge \delta v_2}{(\tilde{u}_2 \tilde{v}_2 - a_2^{-1})^2} = \delta u_4 \wedge \delta v_4 = \delta \tilde{r}_8 \wedge \delta \tilde{s}_8,
\end{equation}
where we use $\delta$ to indicate the exterior derivative on the fibre $E_t$, so these equalities hold under the restrictions of the gluing to the fibres, and $t$ is treated as a constant in the calculations.

\begin{remark}
We have used the same charts $(\tilde{u}_2, \tilde{v}_2)$ and $(\tilde{u}_4, \tilde{v}_4)$ we introduced when constructing the surfaces $\tilde{\X}_t$, but  we require an adjustment to our charts covering $E_8$ in order to define the Hamiltonian structure using the Lemma above. 
If we introduce the charts $(\tilde{u}_8, \tilde{v}_8)$ and $(\tilde{U}_8, \tilde{V}_8)$ in the usual way to cover the exceptional divisor $E_8$ arising from the blowup of $p_8$, namely according to
\begin{equation}
(\tilde{u}_7 - t^2 +a_0, \tilde{v}_7) = (\tilde{u}_8 \tilde{v}_8 , \tilde{v}_8 ) = (\tilde{V}_8, \tilde{U}_8 \tilde{V}_8),
\end{equation}
then the part of $E_8$ away from the inaccessible divisors is visible in the affine $(\tilde{u}_8, \tilde{v}_8)$-chart. 
In this chart, the symplectic form is given by
\begin{equation}
\tilde{\omega}_t = \frac{ \delta \tilde{v}_8 \wedge \delta \tilde{u}_8}{ \left(1 + t \tilde{v}_8 + (t^2 - a_0) \tilde{v}_8^2 + \tilde{u}_8 \tilde{v}_8^2 \right)^2}.
\end{equation}
The $t$-dependence in the symplectic form written in this chart means that Lemma \ref{symplecticlemma} does not apply here, and indeed there exists no Hamiltonian form for the differential system in this chart with respect to the symplectic form in these coordinates.
We adjust our atlas by using instead the chart $(\tilde{r}_8, \tilde{s}_8)$ to cover the relevant part of the exceptional divisor $E_8$. This comes from the symplectic atlas constructed in \cite{T2} for the Okamoto Hamiltonian form of $\pain{IV}$ \cite{Ok1}, which is related to our standard model \eqref{hamP4} by a scaling of dependent and independent variables. The relevant chart in \cite{T2} is denoted by $(x(\infty \infty), y(\infty \infty))$, and is related to ours by
\begin{equation}
x(\infty \infty) = \frac{\tilde{s}_8}{\sqrt{2}}, \qquad y(\infty \infty) = \sqrt{2} \tilde{r}_8.
\end{equation}
%
The part of the exceptional divisor $E_8$ away from its intersection with $\tilde{D}_{\operatorname{red}}$ is visible in this chart, given by $\tilde{s}_8 = 0$, parametrised by $\tilde{r}_8$.
\end{remark}

From Lemma \ref{symplecticlemma} this atlas allows $H(q,p,t)$ to determine Hamiltonians $H_2(\tilde{u}_2, \tilde{v}_2, t)$, $H_4(\tilde{u}_4, \tilde{v}_4, t)$ and $H_8(\tilde{r}_8, \tilde{s}_8, t)$ in the other charts, and we have a global Hamiltonian structure of the differential equation on $\tilde{E}$ provided by the 2-form
\begin{equation}
\begin{aligned}
\tilde{\Omega} &= dp \wedge dq - dH \wedge dt = \frac{d u_2 \wedge d v_2}{(\tilde{u}_2 \tilde{v}_2 - a_2^{-1})^2} - dH_2 \wedge dt \\
&=  d\tilde{u}_4 \wedge d\tilde{v}_4 - dH_4 \wedge dt =  d\tilde{r}_8 \wedge d \tilde{s}_8 - dH_8 \wedge dt,
\end{aligned}
\end{equation}
where $d$ denotes the exterior derivative on the total space $\tilde{E}$, and the Hamiltonians in each chart can be taken as
\begin{equation}
\begin{aligned}
H(q,p,t) &= q p (p-q-t) - a_1 p - a_2 q, \\
H_2(\tilde{u}_2,\tilde{v}_2,t) &= \frac{a_2(t - a_2^2 u_2) + (a_0 -1 - t a_2^2 u_2)v_2 + a_2(2a_1+a_2)u_2 v_2^2 - a_1 a_2^2 u_2^2 v_2^3}{(1-a_2 u_2 v_2)^2}, \\
H_4(\tilde{u}_4,\tilde{v}_4,t) &= \tilde{u}_4 - (a_1 + \tilde{u}_4 \tilde{v}_4)(t + \tilde{v}_4 (1 - a_0 + \tilde{u}_4 \tilde{v}_4)), \\
H_8(\tilde{r}_8, \tilde{s}_8,t) &= -t - \tilde{r}_8 +(1-a_1) \tilde{s}_8 + (a_2 - \tilde{r}_8 \tilde{s}_8)(t - \tilde{s}_8(1+a_0 + \tilde{r}_8 \tilde{s}_8)).
\end{aligned}
\end{equation}
From the second part of Lemma \ref{symplecticlemma}, we know that the system of first-order differential equations is given in charts as follows:
\begin{equation}
\begin{gathered}
q' = \frac{\partial H}{\partial p}, \quad p' = - \frac{\partial H}{\partial q}, \\
\tilde{v}_2'  = (\tilde{u}_2 \tilde{v}_2 - a_2^{-1})^2 \frac{\partial H_2}{\partial \tilde{u}_2}, \quad \tilde{u}_2'   = - (\tilde{u}_2 \tilde{v}_2 - a_2^{-1})^2\frac{\partial H_2}{\partial \tilde{v}_2}, \\ 
\tilde{v}_4' = \frac{\partial H_4}{\partial \tilde{u}_4}, \quad \tilde{u}_4' = - \frac{\partial H_4}{\partial \tilde{v}_4},  \\
\tilde{s}_8' = \frac{\partial H_8}{\partial \tilde{r}_8}, \quad \tilde{r}_8' = - \frac{\partial H_8}{\partial \tilde{s}_8}.
\end{gathered}
\end{equation}
We have the following relation between these Hamiltonians under the gluing of the total space $\tilde{E}$:
\begin{equation}
H = H_2 = H_4 = H_8 - q.
\end{equation}
The relevance of this to the present method is through the uniqueness results of Hamiltonian systems on Okamoto's spaces. The result for $\pain{IV}$ provided in \cite{T2} implies that any system given by a collection of rational Hamiltonians holomorphic on $\tilde{E}$ as constructed here must coincide with the standard Hamiltonian form of $\pain{IV}$ in \eqref{hamP4}. This means that an appropriate identification of the space $E$ from the semi-classical Laguerre weight with the standard model $\tilde{E}$ must identify the differential system \eqref{differential} with the standard Hamiltonian form of $\pain{IV}$, as we will see in \autoref{section36}.

\subsection{Identification on the level of the Picard lattice}

We will obtain the identification between the spaces $E$, $\tilde{E}$ by finding a birational mapping $(x,y) \mapsto (q,p)$ which gives an appropriate isomorphism $\X_c \rightarrow \tilde{\X}_t$. 
We will do this by first finding an identification on the level of Picard lattices $\Pic(\X_c)$, $\Pic(\tilde{\X}_t)$ which
\begin{enumerate}
\item matches the surface roots in \autoref{fig:d-roots-hw} and \autoref{fig:d-roots-P4}, i.e. $\delta_i = \tilde{\delta}_i$ for $i=0, \dots, 6$;
\item preserves the intersection form;
\item preserves effectiveness of divisor classes.
\end{enumerate}
Such an identification is essentially a change of basis for the rank 10 Lorentzian lattice which identifies the surface sublattices, so this is a linear algebra problem with the added consideration of effectiveness of divisor classes.

\begin{lemma} \label{initident}
The following identification of $\Pic(\X_c)$ and $\Pic(\tilde{\X}_t)$ matches the surface root bases in \autoref{fig:d-roots-hw} and \autoref{fig:d-roots-P4}, and preserves the intersection form:
\begin{equation*}
\begin{aligned}
\h_x &= \h_q, 					&&\qquad \h_q = \h_x, \\
 \h_y &= \h_q + \h_p - \E_1 - \E_3,	&&\qquad \h_p = \h_x + \h_y - \F_1 - \F_4, \\
 \F_1 &= \h_q - \E_3, 			&&\qquad \E_1 = \h_x - \F_4,  \\
 \F_2 &= \E_4,	 				&&\qquad \E_2 = \F_3, \\
\F_3 &= \E_2,					&&\qquad \E_3 = \h_x - \F_1, \\
 \F_4 &= \h_q - \E_1, 			&&\qquad \E_4 = \F_2, \\
 \F_5 &= \E_5,					&&\qquad \E_5 = \F_5, \\
  \F_6 &= \E_6, 					&&\qquad \E_6 = \F_6,  \\
  \F_7 &= \E_7, 					&&\qquad \E_7 = \F_7,  \\
  \F_8 &= \E_8,					&&\qquad \E_8 = \F_8.
\end{aligned}
\end{equation*}

\end{lemma}
\begin{proof}
To identify the surface root bases, we require 
\begin{equation}
\begin{gathered}
\F_7 - \F_8 = \E_7 - \E_8,\quad \h_x - \F_3 - \F_4 = \E_1-\E_2, \quad \F_4 - \F_5 = \h_q - \E_1 - \E_5, \\
\F_5 - \F_6 = \E_5 - \E_6, \quad \h_y - \F_4 - \F_5 = \h_p - \E_3 - \E_5, \\
\h_x - \F_1 - \F_2 = \E_3 - \E_4, \quad \F_6 - \F_7 = \E_6 - \E_7.
\end{gathered}
\end{equation}
We begin by setting $\F_8 = \E_8$, since both are effective classes of self-intersection $-1$. 
Then the conditions allow us to recursively determine
\begin{equation}
\F_7 = \E_7, \quad \F_6 = \E_6, \quad \F_5 = \E_5, \quad \F_4 = \h_q - \E_1, \quad \h_y = \h_q + \h_p - \E_1 - \E_3.
\end{equation}
We then make two more choices of classes of $(-1)$-curves to match. In $\Pic(\X_c)$, we have $\F_4$ being a class of self-intersection $-1$ which intersects $D_1$ with multiplicity one, so should be identified with the class of a $(-1)$-curve on $\tilde{\X}_t$ which intersects $\tilde{D}_1$ similarly. 
We take $\F_4 = \E_2$, and by similar reasoning $\F_2 = \E_4$, so are left to determine only $\h_x$, $\F_1$, $\F_3$. 
We require that 
\begin{equation}
\h_x - \F_3 = \E_1, \quad \h_x - \F_1 = \E_3,
\end{equation}
as well as the intersection form being preserved, from which we deduce the rest of the matching. 
From this, computing the inverse is straightforward.
\end{proof}
We remark that at this stage we do not know for certain that the identification in Lemma \ref{initident} preserves effectiveness of divisors. 
If it did not, this would become evident at the next step when we cannot find a birational mapping realising the identification. 
In such a case, the identification would require adjustment.

\subsection{The change of variables and parameter correspondence}
We now seek a birational map which provides an isomorphism which realises the identification obtained above. 
The parameters from the differential systems will be related in a way which can be computed using the period mapping similarly to the standard model of $E_6^{(1)}$-surfaces.
For this reason we now include the parameters $\mathbf{a} = (a_0, a_1, a_2 ; t)$ and $\mathbf{b} = (\alpha, N, n; c)$ in our notation for the surfaces, so the standard model is denoted by $\tilde{\X}_{\mathbf{a}}$ and those from the semi-classical Laguerre weight by $\X_{\textbf{b}}$.
Under our identification in Lemma \ref{initident}, we obtain the choice of symmetry root basis corresponding to that of the standard model in \autoref{fig:symmetryrootsP4}, which we present in \autoref{fig:symmetryrootslaguerre}. 
\begin{figure}[H]
\begin{equation*}\label{eq:d-roots-hw}
	\raisebox{-32.1pt}{\begin{tikzpicture}[
			elt/.style={circle,draw=black!100,thick, inner sep=0pt,minimum size=2mm}]
		\path 	( -1,0) 	node  	(a1) [elt, label={[xshift=0pt, yshift = -25 pt] $\alpha_{1}$} ] {}
		        ( 1,0) 	node  	(a2) [elt, label={[xshift=5pt, yshift = -25 pt] $\alpha_{2}$} ] {}
		        ( 0,1.3) node 	(a0) [elt, label={[xshift=10pt, yshift = -5 pt] $\alpha_{0}$} ] {};
		\draw [black,line width=1pt ] (a0) -- (a1)  -- (a2) -- (a0);
	\end{tikzpicture}} \qquad \qquad
			\begin{alignedat}{2}
			\alpha_{0} &= 2 \h_x + \h_y - \F_1-\F_4 - \F_5 - \F_6 - \F_7 - \F_8, \\
			\alpha_{1} &= \F_1-\F_2, \\
			\alpha_{2} &= \h_y - \F_1-\F_3,
			\end{alignedat}
\end{equation*}
	\caption{The Symmetry Root Basis for the semi-classical Laguerre weight (initial choice).}
	\label{fig:symmetryrootslaguerre}
\end{figure}
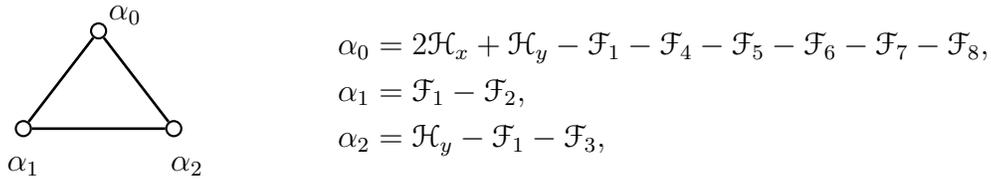
The mechanism by which computation of the root variables for this symmetry root basis for $\X_{\mathbf{b}}$ yields the parameter correspondence for our change of variables is provided by the following fact (see \cite{Sak:2001:RSAWARSGPE, DT18}).

\begin{lemma}
Suppose we have an isomorphism 
\begin{equation}
\iota :  \X_{\mathbf{b}} \longrightarrow \tilde{\X}_{\mathbf{a}} ,
\end{equation}
which induces the identification on the level of Picard lattices in Lemma \ref{initident} so 
\begin{equation}
\iota^*(\tilde{\delta}_i) = \delta_i \text{ for } i=0,\dots, 6, \qquad \iota^*(\tilde{\alpha}_j) = \alpha_j  \text{ for } j=0, \dots, 2.
\end{equation} 
Suppose also that the rational sympectic forms $\omega$, $\tilde{\omega}$ on the surfaces are related by $\iota^* \tilde{\omega} = \omega$. 
Then the definition of the period mapping $\chi$ (using the symplectic form $\omega$) ensures that the root variables for the surface $\X_{\mathbf{b}}$  corresponding to the symmetry root basis in \autoref{fig:symmetryrootslaguerre} coincide with those for $\tilde{\X}_{\mathbf{a}}$:
\begin{equation}
\chi (\alpha_i) = \chi ( \iota^* \tilde{\alpha}_i ) = \tilde{\chi}( \tilde{\alpha}_i) = a_i.
\end{equation}
\end{lemma}

Therefore to compute the parameter correspondence between the system \eqref{differential} from the semi-classical Laguerre weight and the standard Hamiltonian form of $\pain{IV}$ with our identification on the level of Picard lattices, we begin with the computation of the root variables for the surface $\X_{\mathbf{b}}$.
To define the period mapping for the surface $\X_{\mathbf{b}}$ we take the symplectic form as in \eqref{omegaLaguerre} given in relevant charts by
\begin{equation} \label{omegaXcharts}
\begin{aligned}
\omega &= k \frac{dx \wedge dy}{x} = - k \frac{dX \wedge dy}{X} = - k \frac{dx \wedge dY}{Y^2} = k \frac{dX \wedge dY}{X Y^2} \\
&= k \frac{ dV_4 \wedge dU_4}{U_4^2 V_4^2 } = k \frac{dv_5 \wedge du_5}{ u_5^2 v_5^2} = k \frac{dv_6 \wedge du_6}{v_6^2(u_6 v_6 - 2c)^2} \\
&= k \frac{d v_7 \wedge du_7}{v_7 \left( u_7 v_7^2 + 2c(c-1)v_7 - 2c \right)^2}.
\end{aligned}
\end{equation}

\begin{lemma} \label{rootvarlemmalaguerre}

\begin{enumerate}[(a)]
\item The residue of the rational symplectic form $\omega$ given in charts by \eqref{omegaXcharts} along each of the irreducible components of the anticanonical divisor of $\X_{\mathbf{b}}$ is given by
\begin{equation}
\begin{aligned}
\operatorname{res}_{D_0} \omega &= k \frac{d u_7}{4 c^2},  &\quad &\operatorname{res}_{D_1} \omega = -k dy, &\qquad &\operatorname{res}_{D_2} \omega = 0, \qquad \operatorname{res}_{D_3} \omega = 0,  \\
\operatorname{res}_{D_4} \omega &= 0,  &\quad &\operatorname{res}_{D_5} \omega = k dy,  &\quad &\operatorname{res}_{D_6} \omega =  0. 
\end{aligned}
\end{equation}
\item The values of the period mapping on the symmetry roots in \autoref{fig:symmetryrootslaguerre} are given by
\begin{equation}
\chi(\alpha_0) = k \frac{\alpha + n + 1 }{N}, \qquad  \chi(\alpha_1) = - k \frac{\alpha}{N} , \qquad \chi(\alpha_2) = - k \frac{n}{N}.
\end{equation}
\item By normalising the symplectic form so that $k=N$ we have the root variables for the surface $\X_{\mathbf{b}}$ with the choice of symmetry root basis in \autoref{fig:symmetryrootslaguerre} being given by
\begin{equation}
a_0 = \alpha + n + 1, \qquad a_1 = - \alpha, \qquad a_2 = -n.
\end{equation}
In particular, this is consistent with the parameter normalisation $a_0 + a_1 + a_2 =1$ from the Hamiltonian form of $\Pain{IV}$.
\end{enumerate}

\end{lemma}
\begin{proof}
This is again a standard computation, but we present the calculation of $\chi(\alpha_0)$ since it is slightly more involved than others.
According to the usual method we express the symmetry root $\alpha_0$ as a difference of two effective divisors
\begin{equation}
\alpha_0 = [C_0^{1}] - [C_0^0],
\end{equation}
where $C_0^{0} = F_8$ and $C_0^{1}$ is the unique curve on $\X_{\mathbf{b}}$ whose class in $\Pic(\X_{\mathbf{b}})$ is 
\begin{equation}
[C_0^{1}] = 2\h_x + \h_y - \F_1 - \F_4 -\F_5 - \F_6 - \F_7. 
\end{equation}
By a straightforward calculation one finds that this divisor must be given by the proper transform of the curve defined in the affine $(x,y)$-chart by $2c y + (c-1) x + x^2 = 0$.
The component $D_k$ of the anticanonical divisor which intersects $C_0^1$ and $C_0^0$ is $D_0 = F_7 - F_8$, so we have immediately that the point $P_0 = D_0 \cap C_0^0$ is given in coordinates by 
\begin{equation}
P_0 : \quad(u_7, v_7) = \left( -2c\left(N+ c^2 N + 2c(\alpha -N + n + 1)/N  \right), 0 \right).
\end{equation}
To find the point $Q_0 = D_0 \cap C_0^1$, we compute the equation of $C_0^{1}$ in the $(u_7, v_7)$-chart to be
\begin{equation} \label{C01u7v7}
C_0^{1} : \quad 2c(c-1)^2 + u_7 + (c-1) u_7 v_7 = 0,
\end{equation}
so we find the intersection with $D_0 = F_7 - F_8$ by restricting its local equation \eqref{C01u7v7} to $v_7 = 0$ then solving for $u_7$ to obtain
\begin{equation}
Q_0 : \quad(u_7, v_7) = \left( -2 c (c-1)^2, 0 \right).
\end{equation}
We then have 
\begin{equation}
\chi(\alpha_0) =  \int_{P_0}^{Q_0} \operatorname{res}_{D_0} \omega = k \left[ \frac{u_7}{4 c^2} \right]^{u_7 =  -2 c (c-1)^2}_{u_7 = -2c\left(N+ c^2 N + 2c(\alpha -N + n + 1)/N\right)} = \frac{\alpha + n +1 }{N}.
\end{equation}
The other root variables are obtained by similar calculations.
\end{proof}

We now wish to obtain a birational mapping $(x,y) \mapsto (q,p)$ which provides an isomorphism between $\X_{\mathbf{b}}$ and $\tilde{\X}_{\mathbf{a}}$ realising the identification in Lemma \ref{initident}, with parameters $\alpha, n, N$ related to the root variables $a_0, a_1, a_2$ as in Lemma \ref{rootvarlemmalaguerre}. 
The method is identical to that in \cite{DFS19} and in the process we will find the relation between $c$ and $t$ necessary for this isomorphism to exist, which will give the correspondence of independent variables in the differential systems.

We begin by forming an Ansatz for the mapping noting that under our identification we have $\h_q = \h_x$ and $\h_p = \h_x + \h_y - \F_1 -\F_2$, which dictates the degrees of $q$, $p$ as rational functions of $x,y$:
\begin{equation} \label{Ansatz}
q = \frac{a_0 + a_1 x}{b_0 + b_1 x}, \qquad p = \frac{c_{00} + c_{10} x + c_{01} y + c_{11} x y}{d_{00} + d_{10} x + d_{01} y + d_{11} x y}.
\end{equation}
The condition $\h_p = \h_x + \h_y - \F_1 -\F_2$ means that $p$ should provide an affine coordinate on a pencil of curves on $\p^1 \times \p^1$ passing through $q_1$ and $q_2$, so in particular the numerator and denominator of the rational function giving $p$ in \eqref{Ansatz} should be indeterminate at both of these points.
This gives 
\begin{equation}
c_{00} = d_{00} = 0, \qquad c_{11}=d_{11} = 0.
\end{equation}
We now successively impose conditions from the identification to determine the rest of the coefficients, as well as confirm our parameter matching. 
There are many possible orders in which the conditions can be imposed, but we illustrate one set of choices that leads to fairly simple calculations for reasons we outline in the process.
Firstly, the matching of the surface roots $\delta_1, \tilde{\delta}_1$, namely $\h_x - \F_3 - \F_4 = \E_1 - \E_2$, means that the line $X=0$ should be sent under the birational mapping to $p_1: (Q,p) = (0,0)$. 
Rewriting the updated Ansatz in the relevant charts we have 
\begin{equation}
Q = \frac{b_0 X + b_1 }{a_0 X  + a_1}, \qquad p = \frac{c_{10}  + c_{01}X y }{ d_{10}  + d_{01} X y },
\end{equation}
so we require $b_1 = 0$, $c_{10} = 0$. 
Similarly, the matching of $\delta_5, \tilde{\delta}_5$ requires $\h_x - \F_1 - \F_2 = \E_3 - \E_4$, so the line $x=0$ should be sent to $p_3 : (q,P) = (0,0)$, which leads to $a_0 = 0$, $d_{01} = 0$.
Our refined Ansatz for the birational mapping is then
\begin{equation} \label{refinedAnsatz1}
q = A x, \qquad p = B \frac{ y }{x},
\end{equation}
where we have relabeled $A = a_1 / b_0$, $B = c_{01} / d_{10}$. 
We confirm that this gives the desired isomorphisms between components of the anticanonical divisors by calculations in charts. 
For example, to confirm the isomorphism between $E_1 - E_2$ and $H_x - F_3 -F_4$ (the proper transform of $\{X=0\}$), we rewrite \eqref{refinedAnsatz1} in charts $(X,y)$ for $\X_{\mathbf{b}}$ and $(\tilde{u}_1,\tilde{v}_1)$ for $\tilde{\X}_{\mathbf{a}}$, which gives
\begin{equation}
\tilde{u}_1 = \frac{1}{A B y}, \qquad \tilde{v}_1 = B X y. 
\end{equation}
Setting $X=0$ here gives $\tilde{v}_1=0$ with $\tilde{u}_1$ parametrised as a fractional-linear function of $y$, so we have the two copies of $\p^1$ in bijection.

We next check the matching of some other components of the anticanonical divisors, for example $\delta_4 = \h_y - \F_4 - \F_5 = \h_p - \E_3 - \E_4 = \tilde{\delta}_4$ requires the proper transforms of $\{Y=0\}$ and $\{P=0\}$ to be in bijection. 
Calculations in relevant charts reveal that the mapping \eqref{refinedAnsatz1} already provides this, and similarly for the surface roots $\delta_2, \delta_3$.
Matching $\delta_6$ and $\tilde{\delta}_6$ requires isomorphism between $F_6 - F_7$ and $E_6 - E_7$, so the point $q_6 : (u_5, v_5) = (-2c, 0)$ should be sent to $p_6 : (\tilde{u}_5, \tilde{v}_5) = (1,0)$. 
Rewriting \eqref{refinedAnsatz1} in the relevant charts we have 
\begin{equation}
u_5 = \frac{B}{A \tilde{u}_5}, \qquad v_5 = A \tilde{u}_5 \tilde{v}_5,
\end{equation}
so we require $B = - 2 c A$. 
Similarly from the matching of $\delta_6$ we require $q_7$ to be sent to $p_7$, which gives $A = 1/(c-1)$, so we have determined all coefficients in the Ansatz:
\begin{equation} \label{refinedAnsatz2}
q = \frac{t x}{c-1}, \qquad p = \frac{ 2c t y}{(1-c) x}.
\end{equation}

At this stage we have refined the mapping so that it provides an isomorphism everywhere between the surfaces, and now have only to check that the exceptional divisors $F_2, F_3, F_8$ are in bijection with $E_4, E_2, E_8$ respectively. 
For example, to check that $q_2 : (x,y) = (0, \alpha / N)$ is sent to $p_4 : (\tilde{u}_3, \tilde{v}_4) = (a_1, 0)$ as required, we rewrite \eqref{refinedAnsatz2} in relevant charts to find
\begin{equation}
\tilde{u}_3 = - \frac{2 c t^2 y }{(c-1)^2}, \qquad \tilde{v}_3 = \frac{(1-c) x}{2 c t y},
\end{equation}
so after substituting $q_2, p_4$ we require that 
\begin{equation}
a_1 = - \frac{2 c t^2 \alpha }{(c-1)^2 N},
\end{equation}
which with the matching of parameters and root variables in Lemma \ref{rootvarlemmalaguerre} leads to the relation 
\begin{equation} \label{indepmatching}
2c t^2=(c-1)^2 N.
\end{equation}
Similar calculations for $F_3, F_8$ reveal that with the root variables in Lemma \autoref{rootvarlemmalaguerre} as well as the matching \eqref{indepmatching} of $t$ and $c$, the correspondence \eqref{refinedAnsatz2} provides an isomorphism between surfaces $\X_{\mathbf{b}}$ and $\tilde{\X}_{\mathbf{a}}$.
This birational mapping provides a change of variables which identifies our differential system with the standard Hamiltonian form of $\pain{IV}$, which can be verified by direct calculation.

\begin{theorem} \label{changeofvarslaguerre}
The following change of variables identifies the differential system  \eqref{differential} from the semi-classical Laguerre weight with the standard Hamiltonian form \eqref{hamP4} of the fourth Painlev\'e equation:
\begin{equation}
\frac{x(c)}{c-1} = \frac{q(t)}{t}, \qquad  \frac{2c}{(c-1)} \frac{ y(c)}{x(c)}= - \frac{p(t)}{t}, \qquad 2c t^2=(c-1)^2 N,
\end{equation}
with parameters related by
\begin{equation}
a_0=\alpha + n + 1, \qquad a_1=-\alpha, \qquad  a_2=-n.
\end{equation}
\end{theorem}

\begin{remark}
The relation between the independent variables $c$ and $t$ is consistent with the following well-known symmetry  $q \rightarrow -q$, $p \rightarrow -p$, $t \rightarrow -t$ of $\pain{IV}$, with the parameters $a_i$ unchanged.
\end{remark}

\begin{remark}
We could also use the methods of \cite{DFS19} and adjust the change of variables in \autoref{changeofvarslaguerre} by an element of the symmetry group $W(A_2^{(1)}) \rtimes \operatorname{Aut}(A_2^{(1)})$ of the surfaces such that the discrete dynamics from the system \eqref{discrete1} is matched with that of a standard discrete Painlev\'e equation on the same standard model of $E_6^{(1)}$-surfaces.
This will be a B\"acklund transformation of the differential system, so that the discrete and differential systems are simultaneously identified with the standard ones.
We can see the need for such an adjustment by noting that compared with \cite{KajNouYam:2017:GAOPE}, the evolution in $n$ of the root variables in \autoref{changeofvarslaguerre} does not match with that of the standard discrete Painlev\'e equation ($a_0$ should not depend on $n$). 
Therefore, although we have identified the differential system with the fourth Painlev\'e equation, the change of variables does not work for the discrete equations. 
Adjusting our identification by an appropriate symmetry we have another transformation:
\begin{equation}
\begin{aligned}
\frac{x_n}{c-1}&=\frac{q_n(q_n^2+(t-p_n)q_n+\alpha)}{t(q_n^2+(t-p_n)q_n-n-1)},\\
-\frac{2c y_n}{(c-1)x_n}&=\frac{q_n^3 p_n+(2+2n+\alpha+t p_n-p_n^2)q_n^2+(n+1)(t-2p_n) q_n-(n+1)^2}{t q_n(q_n^2+(t-p_n)q_n-n-1)},
\end{aligned}
\end{equation}
which identifies the discrete system \eqref{discrete} with the standard d-$\dPain{A_2^{(1)}/E_6^{(1)}}$ \cite{KajNouYam:2017:GAOPE} system
\begin{equation}
\begin{aligned}
q_{n+1}+q_n &=p_n-t-a_2/p_n, \\ 
p_n+p_{n-1} &=q_n+t+a_1/q_n,   
\end{aligned}
\end{equation}
with $a_0=-\alpha,$ $a_1=-1-n$, $a_2=2+n+\alpha$, $2ct^2 = (c-1)^2 N$. 
Since we have adjusted our identification in \autoref{changeofvarslaguerre} by a symmetry of $\pain{IV}$, the differential systems will also be identified.
\end{remark}

\subsection{Hamiltonian structure and justification of equivalence} \label{section36}

Since the isomorphism obtained above identifies the inaccessible divisors on the surfaces $\X_{\mathbf{b}}$ and $\tilde{\X}_{\mathbf{a}}$, this provides an isomorphism between the bundles (with parameters related as in \autoref{changeofvarslaguerre})
\begin{equation}
\varphi : E \longrightarrow \tilde{E}.
\end{equation}
Restricting to the fibres $E_t = \X_{\mathbf{b}} \backslash D_{\operatorname{red}}$, $E_c = \tilde{\X}_{\mathbf{a}} \backslash \tilde{D}_{\operatorname{red}}$, this transformation can be verified by direct calculation to be symplectic with respect to the normalised 2-forms $\omega$, $\tilde{\omega}$ in charts:
\begin{equation}
 \varphi_t^*( \delta p \wedge \delta q) = N \frac{\delta x \wedge \delta y}{x} ,
\end{equation}
where again $\delta$ is the exterior derivative on the fibre, so $t$ and $c$ are treated as constants in the calculation.
Noting that the symplectic forms are independent of $t$, $c$, from Lemma \ref{symplecticlemma} we automatically obtain the Hamiltonian structure of the system from the Laguerre weight:
\begin{theorem}
The differential system \eqref{differential} can be written in the Hamiltonian form 
\begin{equation}
\begin{gathered}
\frac{N}{x} \frac{dy}{dc} = \frac{\partial K}{\partial x }, \qquad \frac{N}{x} \frac{dx}{dc} = - \frac{\partial K}{\partial y }, \\
K(x,y,c) = \frac{N}{4 c^2} \left( \left((c^2-1)N - 2c \right) y + (c+1)(N y + n) x + 2c(c+1)\frac{ y (N y - \alpha)}{x} \right).
\end{gathered}
\end{equation}
Under the identification $\varphi$ of the total spaces $E$, $\tilde{E}$ we have the equality of 2-forms
\begin{equation}
N \frac{dx \wedge dy}{x} - d K \wedge d c = dp \wedge dq - dH \wedge d t,
\end{equation}
where $d$ is the exterior derivative on the total space, and $H$ is the Hamiltonian for $\pain{IV}$ as in system \eqref{hamP4}.
\end{theorem}

By pulling back the symplectic atlas for the standard model of Okamoto's space under our isomorphism to the bundle $E$, we find that the differential system \eqref{differential} indeed has a Hamiltonian structure holomorphic on $E$ and extending meromorphically to its closure (as each of the Hamiltonians is rational).
Thus by the uniqueness result of \cite{T2} our isomorphism is guaranteed to identify the differential systems as we have seen in \autoref{changeofvarslaguerre}.

\section{Hypergeometric weight} \label{section4}

We now turn to the case of the system of differential equations \eqref{syst P6} from the hypergeometric weight \eqref{eq:hyp-weight}. 
The surfaces are the same as those constructed in \cite{DFS19} in the analysis of the discrete system \eqref{(2.3)}-\eqref{(2.4)}, and we recall the point configuration, surface root basis, and symplectic structure in Appendix \autoref{AppendixA}
, in order for the present paper to be self-contained.

We will recycle notation from the previous section, denoting the surfaces constructed from the system \eqref{syst P6} in $(x,y)$-coordinates by $\X_c$, irreducible components $D_k$ of the anticanonical divisor, with their union again being $D_{\operatorname{red}} = \bigcup_{k} D_k$.
The construction of the space of initial conditions gives the bundle 
\begin{equation} \label{bundlehypergeom}
\begin{gathered}
\rho : E \longrightarrow B, \\
\rho^{-1}(c) = \X_c \backslash D_{\operatorname{red}},
\end{gathered}
\end{equation}
where $B= \C \backslash \{0, 1 \}$ is the independent variable space for system \eqref{syst P6}, with the fixed singularities removed.
This again admits a uniform foliation by solution curves of the system \eqref{syst P6} transverse to the fibres.

We also recycle notation for the standard model of the relevant Sakai surfaces, namely those of surface type $D_4^{(1)}$ as presented in \cite{KajNouYam:2017:GAOPE}, which we recall in Appendix \autoref{AppendixB}. 
These provide spaces of initial conditions for a Hamiltonian system for $(f(t), g(t))$ equivalent to $\pain{VI}$, and we denote them by $\tilde{\X}_t$, with the inaccessible divisors denoted by $\tilde{D}_k$, their union being $\tilde{D}_{\operatorname{red}}$, and bundle 
\begin{equation}
\begin{gathered}
\tilde{\rho} : \tilde{E} \longrightarrow \tilde{B}, \\
\tilde{\rho}^{-1}(t) = \tilde{\X}_t \backslash \tilde{D}_{\operatorname{red}},
\end{gathered}
\end{equation}
where $\tilde{B} = \C \backslash \{0,1\}$ is obtained by removing the locations of fixed singularities of $\pain{VI}$ from the complex plane.

\subsection{Inaccessible singular points}

The calculation methods required to construct a space of initial conditions for the system \eqref{syst P6} are more involved than in \autoref{section3}. 
In particular, careful consideration has to be made to determine whether an indeterminacy of the vector field constitutes a singularity which requires blowing up in order to regularise the system.

As usual we take the phase space of the system \eqref{syst P6} initially to be the trivial bundle over $B = \C \backslash \{0, 1\}$ with fibre $\p^1 \times \p^1$ covered by affine charts $(x,y), (X,y), (x,Y), (X,Y)$ where $X =1/x, Y=1/y$. 
We begin by looking for points of indeterminacy of the vector field in the $(x,y)$-chart. 
Looking firstly for indeterminacies of the rational function $x'(c)$ we find precisely the points $q_1,\,q_2,\,q_3,\,q_4$ as in Appendix \autoref{AppendixA}.
Each of these singularities are resolved through exactly one blowup, which can be verified by computations in charts as usual.

However, when looking for indeterminacies of $y'(c)$ we also find an extra point, which we call $q_9$, given by 
\begin{equation}
q_9 : (x,y ) = \left( \frac{n+\alpha+\beta}{2}, -\frac{ n^2+2n(\alpha+\beta)+(\alpha-\beta)^2}{4} \right).
\end{equation}
The point $q_9$ lies on the curve $\Gamma$ (\ref{curve}) along with other points $q_1,\,q_2,\,q_3,\,q_4$, which is to be expected since $\Gamma$ is the vanishing locus of the denominators of both components of the vector field in the affine $(x,y)$-chart. 
We see that while $y'(c)$ has an indeterminacy at $q_9$, $x'(c)$  diverges (i.e., the denominator of $x'(c)$  vanishes while the numerator stays nonzero). We claim that this is an indication that this extra point is an inaccessible singularity: an indeterminacy of (one or more components of) the vector field, but with no solution starting from a regular point passing through. By regular here we mean where the vector field is regular, so classical ODE theorems guarantee existence and uniqueness of analytic solutions through these points. 

In the $(X,Y)$-chart we see the same points $q_5,\,\ldots,\,q_8$ as in Appendix \autoref{AppendixA}, and standard calculations show that the vector field is regular on the affine part of $F_8$ away from its intersection with the anticanonical divisor.
However, after the blowup of $q_6$ we see a second extra point on the exceptional line in addition to $q_7$, given by 
\begin{equation}
q_{10} : (u_6, v_6)  = (1, 0).
\end{equation}

We claim that these two extra points $q_9$ and $q_{10}$ are inaccessible singularities, so can be removed as part of the inaccessible divisors and we have the bundle $E$ as in \eqref{bundlehypergeom}.
We first consider the extra point $q_9$. 
If we assume that the  solutions of   (\ref{syst P6}) in the neighbourhood of the curve $\Gamma$ are single-valued and given by Taylor series around some $c_0 \in B$ (here we assume the Painlev\'e property), we can see that no solutions pass through the extra point $q_{9}$ transverse to the fibre over $c_0$. 
Indeed, let 
\begin{gather*}
x(c)= \sum_{i=0}^{\infty} a_i (c-c_0)^i,\;\; 
y(c)=\sum_{i=0}^{\infty} b_i (c-c_0)^i
\end{gather*}
with $(x(c_0),\,y(c_0))=(a_0,\,b_0)$  on the curve $\Gamma$, so these satisfy the relation 
\begin{equation}
b_0=\alpha\beta-a_0 (\alpha+\beta)-n a_0+a_0^2.
\end{equation}
Clearing denominators in (\ref{syst P6}), then substituting the above expansions in $c$ near $c_0$ we have
\begin{gather*}
(a_0-1)(a_0-\alpha)(a_0-\beta)(a_0-\gamma) +O(c-c_0)=0,\\
(a_0-1)(a_0-\alpha)(a_0-\beta)(a_0-\gamma)(n+\alpha +\beta -2a_0)  +O(c-c_0)=0.
\end{gather*}
These equalities of asymptotic expansions about $c_0$ may only be both balanced if $a_0=1,\,\alpha,\,\beta,\,\gamma$, i.e. such solutions may only pass through one of $q_1,\,\ldots,\,q_4$, and not the extra basepoint  $q_9$ at which $a_0 = (n + \alpha + \beta)/2$. While the second equality is balanced by $q_9$, the first is not, which reflects the fact that while $y'(c)$ is indeterminate at $q_9$, $x'(c)$  diverges.

We now consider the second extra point $q_{10}$, so we rewrite system \eqref{syst P6} in the $(u_6,v_6)$-chart, which recall is defined by
\begin{equation}
x=\frac{1}{v_6}, \quad y =\frac{1}{u_6 v_6^2}.
\end{equation}
Computing the system in this chart we see that at $q_{10}$ the denominator of $v_6'(c)$  vanishes while the numerator stays finite and nonzero. 
Again, this is an indication that this extra point is inaccessible. 
We proceed as before: if we assume that the only solutions of the system in the neighbourhood of the exceptional divisor $F_6$ are single-valued (assuming the Painlev\'e property), we can see that the only point they may pass through is $q_7$, and no solutions pass through the extra point $q_{10}$.  
Consider a solution given by an expansion about $c=c_0$:
\begin{gather*}
u_6(c)= \sum_{i=0}^{\infty} a_i (c-c_0)^i,\;\; 
v_6(c)=\sum_{i=1}^{\infty} b_i (c-c_0)^i.
\end{gather*}
After finding the rational functions $u_6', v_6'$, we substitute these expansions and obtain
\begin{gather*}
- 2 (a_0 - 1)(c_0 - (c_0 - 1)a_0)+ O(c-c_0)=0,\\
(c_0 - 1)(b_1 c_0 - 1)a_0^2 - c_0(2-(c_0-1)b_1)a_0 + c_0 +O(c-c_0)=0.
\end{gather*}
We see that the first line is balanced when $a_0 = 1$, which corresponds to $q_{10}$, but upon substitution of this into the second line the leading term is equal to $-1$, so again we see that $q_{10}$ is an inaccessible singularity.
 
To summarize, in order to arrive at a space of initial conditions we need to differentiate between points where one or more
components of the vector field diverge and points where either both components have
indeterminacies or one is indeterminate and the other is regular. So we need to identify which of the points   have the following properties:
\begin{enumerate}[(i)]
\item the vector field corresponding has an  indeterminacy of both components;

\item the vector field  has an  indeterminacy in one  component, but is finite in the other;

\item the vector field  has an  indeterminacy in one  component but diverges in the other.

\end{enumerate}
Only Case (i) or Case (ii) can possibly constitute accessible singularities, while Case (iii) can be neglected, and
  these points may be removed as part of the inaccessible divisors.

\begin{remark}
There are also other types of indeterminacies of differential systems which do not require blowing up to arrive at a space of initial conditions, though in our experience these arise only in rather artificial examples.
For example, consider the system 
\begin{equation} \label{betagamma}
\begin{aligned}
\zeta' &= \frac{ (1+\zeta)(a_1 + a_2 + a_1 \zeta) + 2t(1+\zeta) \eta - 3 \zeta \eta^2 }{\eta}, \\
\eta' &= \frac{a_2(1+\zeta) + t (1+\zeta) \eta + (1-\zeta) \eta^2} {1+\zeta},
\end{aligned}
\end{equation}
for $\zeta(t), \eta(t)$, which is equivalent to $\pain{IV}$ in the Hamiltonian form \eqref{hamP4} via the birational transformation
\begin{equation}
q = \frac{ \eta }{ 1+ \zeta}, \qquad p = \eta.
\end{equation}
We see that $\zeta'$ given by \eqref{betagamma} is indeterminate when $( \zeta, \eta) = ( -1 - a_2 / a_1, 0)$, while $\eta'$ is regular. 
However, this corresponds in $(q,p)$-coordinates to a point where the Hamiltonian system \eqref{hamP4} is regular, and there is no infinite family of solutions passing through it at a fixed value of $t$, so no blowup is required.
This is also the case for the system for $\zeta, \eta$, which can be seen through classical Painlev\'e analysis to establish that this indeterminacy does not require blowing up since there is no infinite family of solutions to separate in order to achieve a uniform foliation. 
\end{remark}

\subsection{The change of variables from initial identification on the level of the Picard lattice}

In \cite{DFS19} it was shown in detail how to transform the discrete system  (\ref{(2.3)})-(\ref{(2.4)}) to the standard form of a discrete Painlev\'e equation of surface type $D_4^{(1)}$ as it appears in \cite{KajNouYam:2017:GAOPE}.  
The change of variables
\begin{equation}\label{transf}
\begin{gathered}
f=-\frac{(x-\beta)(x-\gamma)}{c(n x-\alpha\beta+(\alpha+\beta)x-x^2 +y)},\\
g=-\frac{(x-\gamma)(\beta(\gamma-\beta-\alpha c)+(\alpha+\beta)c x-\gamma x-(c-1)x^2+n(x-\beta+c x)+cy)}{(x-\beta)(x-\gamma)+c(n x-\alpha\beta+(\alpha+\beta)x-x^2+y)},
\end{gathered}
\end{equation}
will not only transform the discrete system to the standard form, but also identify the differential system \eqref{syst P6} with a first-order system equivalent to the sixth Painlev\'e equation. 
The result is the following, which is proved by direct calculation.

\begin{theorem}\label{theorem P6 final} \cite{DFS19}
Let $x=x(c)$ and $y=y(c)$ solve (\ref{syst P6}). Take $t=1/c$ and define functions $f=f(t)$, $g=g(t)$  by  (\ref{transf}). 
Then  system (\ref{syst P6})  is transformed to system
\begin{eqnarray}\label{fgfinal}
t(t-1)f'&=&(1+2n+\alpha+2\beta-2\gamma+2g)f^2+(\beta-\gamma+2g)t\\\nonumber
&&-f(1+n(1+ t)+\alpha+\beta+2t \beta-2\gamma-t\gamma+2(1+t)g),\\\nonumber
t(t-1)f g'&=&t (\beta-\gamma+g)g-(1+n+\beta-\gamma+g)(n+\alpha+\beta-\gamma+g)f^2,
\end{eqnarray}
which is equivalent to the sixth Painlev\'e equation for the function $f$ given by
\begin{eqnarray}\label{eq P6} 
f'' &=&\frac{1}{2}\left(\frac{1}{f}+\frac{1}{f-1}+\frac{1}{f-t}\right)f'^2-\left(\frac{1}{t}+\frac{1}{t-1}+\frac{1}{f-t}\right)f'\\\nonumber
&&+\frac{f(f-1)(f-t)}{t^2(t-1)^2}\left(A+B \frac{ t}{f^2}+C \frac{t-1}{(f-1)^2}+D\frac{t(t-1)}{(f-t)^2}\right),
\end{eqnarray}
 with parameters 
\begin{equation}\label{par P6}
A=\frac{1}{2}(\alpha-1)^2,\,\,B=-\frac{1}{2}(\beta-\gamma)^2,\,\,C=\frac{1}{2}(n+\beta)^2,\,\,D=\frac{1}{2}-\frac{1}{2}(n+\alpha-\gamma)^2.
\end{equation}
\end{theorem}

However, we could also use the initial identification on the level of Picard lattices described in  \cite{DFS19} to find a birational isomorphism which provides a transformation from the differential system  (\ref{syst P6}) to the standard form of the sixth Painlev\'e equation. 
This initial identification is provided by the following Lemma, which can be verified by direct calculation. 


\begin{lemma}\cite{DFS19}\label{lem:change-basis-pre}
	The following identification of $\Pic(\X_c)$ and $\Pic(\tilde{\X}_{t})$ matches the surface root bases in \autoref{surfacerootshypergeom} and \autoref{surfacerootsP6}, and preserves the intersection form:
	\begin{align*}
		\mathcal{H}_{x} & = \mathcal{H}_{g}, & \qquad 
		\mathcal{H}_{f} & = 2\mathcal{H}_{x} + \mathcal{H}_{y} - \mathcal{F}_{3} - \mathcal{F}_{4} - \mathcal{F}_{5} - \mathcal{F}_{6},\\
		\mathcal{H}_{y} & = \mathcal{H}_{f} + 2\mathcal{H}_{g}  - \mathcal{E}_{3} -\mathcal{E}_{4} - \mathcal{E}_{5} - \mathcal{E}_{6}, & 
		\qquad 	\mathcal{H}_{g} &= \mathcal{H}_{x},\\
		\mathcal{F}_{1} & = \mathcal{E}_{1}, & \qquad \mathcal{E}_{1} & = \mathcal{F}_{1},\\
		\mathcal{F}_{2} & = \mathcal{E}_{2}, & \qquad \mathcal{E}_{2} & = \mathcal{F}_{2},\\
		\mathcal{F}_{3} & = \mathcal{H}_{g} - \mathcal{E}_{6}, & \qquad \mathcal{E}_{3} & = \mathcal{H}_{x} - \mathcal{F}_{6},\\
		\mathcal{F}_{4} & = \mathcal{H}_{g} - \mathcal{E}_{5}, & \qquad \mathcal{E}_{4} & = \mathcal{H}_{x} - \mathcal{F}_{5},\\
		\mathcal{F}_{5} & = \mathcal{H}_{g} - \mathcal{E}_{4}, & \qquad \mathcal{E}_{5} & = \mathcal{H}_{x} - \mathcal{F}_{4},\\
		\mathcal{F}_{6} & = \mathcal{H}_{g} - \mathcal{E}_{3}, & \qquad \mathcal{E}_{6} & = \mathcal{H}_{x} - \mathcal{F}_{3},\\
		\mathcal{F}_{7} & = \mathcal{E}_{7}, & \qquad \mathcal{E}_{7} & = \mathcal{F}_{7},\\
		\mathcal{F}_{8} & = \mathcal{E}_{8}, & \qquad \mathcal{E}_{8} & = \mathcal{F}_{8}.
		\end{align*}
\end{lemma}
Using this identification we obtain the choice of symmetry root basis corresponding to the standard one, which we present in \autoref{symmetryrootshypergeom}.
Using this we compute the root variables for this symmetry root basis, which will yield the parameter correspondence between the systems under this identification.


\begin{lemma} \label{rootvarsinit}
\begin{enumerate}[(i)]
\item Using the symplectic form \eqref{omegahypergeom}, the values of the period mapping on the symmetry root basis in \autoref{symmetryrootshypergeom} are given by
\begin{equation}
\begin{gathered}
\chi(\alpha_0) = k(\gamma - \alpha - n), \qquad \chi(\alpha_1) = k(\alpha -1), \qquad \chi(\alpha_2) = k(1- \gamma), \\
 \chi(\alpha_3) = k(\beta + n), \qquad \chi(\alpha_4) = k(\gamma - \beta).
\end{gathered}
\end{equation}
\item Normalising the symplectic form \eqref{omegahypergeom} to $k=1$, the root variables for the symmetry root basis in \autoref{symmetryrootshypergeom} are given by
\begin{equation}
a_0 = \gamma - \alpha - n, \quad a_1 = \alpha -1, \quad a_2 = 1- \gamma, \quad a_3 = \beta + n, \quad a_4 = \gamma - \beta.
\end{equation}

\end{enumerate}
\end{lemma}

 \begin{figure}[H]
 \begin{equation*}  
	\raisebox{-32.1pt}{\begin{tikzpicture}[
			elt/.style={circle,draw=black!100,thick, inner sep=0pt,minimum size=2mm}, scale=1]
		\path 	(-1,1) 	node 	(d0) [elt, label={[xshift=-12pt, yshift = -10 pt] $\alpha_{0}$} ] {}
		        (-1,-1) node 	(d1) [elt, label={[xshift=-12pt, yshift = -10 pt] $\alpha_{1}$} ] {}
		        ( 0,0) 	node  	(d2) [elt, label={[xshift=13pt, yshift = -12 pt] $\alpha_{2}$} ] {}
		        ( 1,1) 	node  	(d3) [elt, label={[xshift=12pt, yshift = -10 pt] $\alpha_{3}$} ] {}
		        ( 1,-1) node 	(d4) [elt, label={[xshift=12pt, yshift = -10 pt] $\alpha_{4}$} ] {};
		\draw [black,line width=1pt ] (d0) -- (d2) -- (d1)  (d3) -- (d2) -- (d4);
	\end{tikzpicture}} \qquad
			\begin{alignedat}{2}
			\alpha_{0} &= \h_y - \mathcal{F}_{3} - \mathcal{F}_{4}, &\qquad  \alpha_{3} &= 2\h_x - \h_y -\F_{345678},\\
			\alpha_{1} &= \F_1 - \F_2, &\qquad  \alpha_{4} &= \F_3-\F_4.\\
			\alpha_{2} &= \mathcal{F}_{4} - \mathcal{F}_{1},
			\end{alignedat}
\end{equation*}
\caption{The initial choice of Symmetry Root Basis for the hypergeometric weight}
 \label{symmetryrootshypergeom}
 \end{figure}
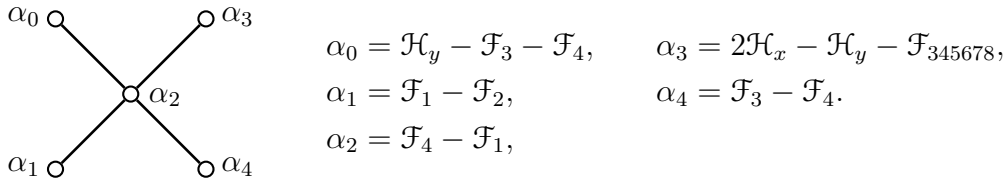

To find the isomorphism between surfaces $\X_{\mathbf{b}}$ and $\tilde{\X}_{\mathbf{a}}$ which realises the initial identification in Lemma \ref{lem:change-basis-pre}, the technique is the same as was used in \autoref{section3} for the semi-classical Laguerre weight (see also \cite{DFS19} for the isomorphism realising the final adjusted identification that matches the discrete dynamics). 
In the process we find the matching of the independent variables $t$ and $c$ to be given by $c t = 1$, and this again provides a change of variables from the system \eqref{syst P6} to the standard Hamiltonian form of $\pain{VI}$.
\begin{theorem}\label{theorem P6 initial}
Let $x=x(c)$ and $y=y(c)$ solve (\ref{syst P6}). Take $t=1/c$ and define the functions $f=f(t)$, $g=g(t)$  by  
\begin{gather}\label{transf initial}
f=-\frac{(x-\beta)(x-\gamma)}{c(n x-\alpha\beta+(\alpha+\beta)x-x^2 +y)},\;\;\;
g= \gamma-x.
\end{gather}
Then the system (\ref{syst P6})  is transformed to 
\begin{eqnarray}\label{fginitial}
t(t-1)f'&=& (1+\alpha-2\gamma+2g)f^2+t(\beta-\gamma+2g)\\\nonumber
&&+(n(t-1)-1-\alpha-\beta+\gamma(2+t)-2(t+1)g)f ,\\\nonumber
t(t-1)f g'&=&(\gamma-1-g)(\alpha-\gamma+g)f^2+t(\beta-\gamma+g)g ,
\end{eqnarray}
which is precisely the standard Hamiltonian form \eqref{fgP6} of $\pain{VI}$ with parameters
\begin{equation} \label{P6paramsinit}
a_0 =  \gamma - \alpha - n, ~~ a_1 = \alpha - 1 , ~~a_2 = 1 - \gamma, ~~ a_3 = \beta + n, ~~ a_4 = \gamma - \beta.
\end{equation}

\end{theorem}

\begin{remark}\label{remark HJ}
In \cite{HFC19}, a second order second degree differential equation for the function $x=x(c)$  was obtained from (\ref{syst P6}). By taking  a  function $h=h(c)$ defined by  
\begin{equation}\label{tr1}
x=\frac{c(c-1)h'-(n+\beta)h^2+(n-\gamma+(n+\alpha+\beta)c)h-(n+\alpha-\gamma)c}{2(c-1)h},
\end{equation} 
it was shown that this differential equations can be transformed to the sixth Painlev\'e equation  for the function $h$ 
 with parameters 
$$A=\frac{1}{2}(n+\beta)^2,\,\,B=-\frac{1}{2}(n+\alpha-\gamma)^2,\,\,C=\frac{1}{2}(\beta-\gamma)^2,\,\,D=-\frac{1}{2}(\alpha-2)\alpha.$$
It is not difficult to check that by taking further $$h(c)=\frac{c f(1/c)-1}{f(1/c)-1}$$ and $t=1/c$ the function $f(t)$ satisfies the sixth Painlev\'e equation with parameters (\ref{par P6}). The approach in \cite{HFC19} was computational: by taking a suitable Ansatz for the form of the transformation from $x$ to $h$ (as in a general Riccati equation with unknown coefficients divided by a function linear in $h$), one could calculate unknown coefficients using any computer algebra package.  Remark~\ref{remark Ric} explains why such an Ansatz worked.
One could also have taken a different Ansatz: both numerator and denominator are as  in  a general Riccati equation with unknown coefficients, but since initial conditions for $n=0$ should satisfy a Riccati equation, then one can assume that, for instance, the denominator is linear in $h$. This reduces the number of unknown  coefficients and simplies the calculations.
\end{remark}

\subsection{Hamiltonian structure and explanation of equivalence}

The isomorphism between surfaces $\X_{\mathbf{b}}$ and $\tilde{\X}_{\mathbf{a}}$ again provides an isomorphism between the bundles forming Okamoto's space
\begin{equation}
\varphi : E \longrightarrow \tilde{E}.
\end{equation}
Restricting to the fibres $E_t = \X_{\mathbf{b}} \backslash D_{\operatorname{red}}$, $E_c = \tilde{\X}_{\mathbf{a}} \backslash \tilde{D}_{\operatorname{red}}$, this transformation can be verified by direct calculation to be symplectic with respect to the normalised 2-forms $\omega$, $\tilde{\omega}$ in charts:
\begin{equation} \label{deltaP6}
 \varphi_t^*\left(  \frac{\delta g \wedge \delta f}{f} \right) = \frac{\delta y \wedge \delta x}{s(x,y)},
\end{equation}
where 
\begin{equation} \label{s}
s(x,y) = x^2 - (n + \alpha + \beta)x - y + \alpha \beta,
\end{equation}
and again $\delta$ is the exterior derivative on the fibre, so $t$ and $c$ are treated as constants in the calculation.
Since the symplectic form $\omega$ on $\X_{\mathbf{b}}$ is $c$-independent and $\tilde{\omega}$ on $\tilde{\X}_{\mathbf{a}}$ is $t$-independent, Lemma \ref{symplecticlemma} guarantees the existence of a Hamiltonian structure for the differential system \eqref{syst P6}. 

\begin{theorem} \label{covhypergeometric}
System (\ref{syst P6}) can be written as the following non-autonomous Hamiltonian system with respect to the symplectic form $\omega=(dx\wedge dy)/s$, where $s$ is given by \eqref{s}:
\begin{equation}\label{syst K}
\frac{x'}{s}=\frac{\partial K}{\partial y}, \quad \frac{y'}{s}=-\frac{\partial K}{\partial x}, 
\end{equation}
where 
\begin{equation*}
K=\frac{1}{c-1}\left(\frac{(x-1)(x-\alpha)(x-\beta)(x-\gamma)}{c(\alpha\beta-(n+\alpha+\beta)x+x^2-y)}-\frac{x(x+n-1-\gamma)}{c} -y\right).
\end{equation*}
Moreover, we have the following equality of 2-forms under the identification $\varphi$, as well as $(f,g) = (q, qp)$:
$$dp \wedge dq - d H \wedge dt = \frac{dg \wedge df}{f} - d H_1\wedge dt  =  \frac{dy \wedge dx}{s}  - dK\wedge dc$$
with $t=1/c$ and parameters related according to \eqref{P6paramsinit}.
\end{theorem} 

\begin{proof}
Along the lines of Lemma \ref{symplecticlemma}, the condition \eqref{deltaP6} dictates the symplectic form with respect to which our system for $x,y$ should be Hamiltonian, so the function $K(x,y,c)$ will be obtained by solving a system of partial differential equations.
Let us first find a function $K$ satisfying $ \partial K/\partial y=x'/s$, with $x'$ given by the system \eqref{syst P6}. 
Take the  right hand side of the first equation \eqref{syst P6} and divide it by $s$.  Then  integrate the obtained rational expression with respect to $y$. The result is another rational expression in terms of $x,\,y$. Instead of a constant of integration we take the unknown function $F(x)=F(x,c)$. Next we substitute this expression with added $F(x)$ into  the equation $\partial K/\partial x=-  y'/s $, where $y'$ is replaced by the right hand side of the  second equation in  (\ref{syst P6}). This gives a simple differential equation  for $F$, namely, 
$F'(x)(c-1)=1-n+\gamma-2x$, where we regard $c$ as a constant as well and derivative is with respect to $x$. This equation can be solved explicitly as 
$$F(x)=-\frac{x(n-1-\gamma+x)}{c(c-1)}+C,$$ where $C$ is an arbitrary function of only $c$ which can be set to zero. Hence, 
$$
K=\frac{1}{c-1}\left(\frac{(x-1)(x-\alpha)(x-\beta)(x-\gamma)}{c(\alpha\beta-(n+\alpha+\beta)x+x^2-y)}-\frac{x(x+n-1-\gamma)}{c} -y\right).
$$
We could also proceed in the reverse order: first integrate the expression  $-  y'/s $ with respect to $x$ and then find a function $G(y)=G(y,c)$ from $ \partial K/\partial y=x'/s$. This gives an expression for $K$  which differs from the previous only by terms depending on parameters and $c$, but this does not influence \eqref{syst K}.  
 \end{proof}
If we rewrite $K$ in terms of the variables $x$ and $s$,  we see that the resulting expression is holomorphic on the complement of the anticanonical divisor, and extends meromorphically to the part of the surface $\X_{\mathbf{b}}$ visible in the $(x,s)$-chart.
Similarly to in \autoref{section3} we can pull back the symplectic atlas from \cite{T1} to the bundle $E$, and see that the global Hamiltonian structure of the differential system \eqref{syst P6} is given by a collection of rational Hamiltonians holomorphic on $E$, so the uniqueness result of \cite{T1} implies that it should be transformed under our identification to the standard  Hamiltonian form of the sixth Painlev\'e equation, as we have observed.

\begin{remark}\label{remark Ric}
We remark that system (\ref{syst P6}) has special solutions for $n=0$, namely, $y=0$ with $x=x(c)$ satisfying the Riccati equation
\begin{equation}\label{Ric x}
c(c-1)x'+(c-1)x^2+(1-c(\alpha+\beta)+\gamma)x+\alpha\beta c-\gamma=0. 
\end{equation}
The initial condition  $x_0$ for the discrete system  (\ref{(2.3)}) and (\ref{(2.4)}) satisfies this equation.
When we take system \eqref{fgP6} with parameters \eqref{P6paramsinit}, we also have special Riccati solutions for the function $f=f(t)$ with $n=0$:
\begin{equation}\label{Ric f}
t(t-1)f'+(\alpha-1)f^2-(\alpha-\beta-1+t\gamma)f+(\gamma-\beta)=0. 
\end{equation}
The relation between (\ref{Ric x}) and (\ref{Ric f}) is very simple; it is just a fractional-linear transformation, namely, $$f=\frac{x-\gamma}{c(x-\alpha)}$$  with  $t=1/c$. 
We also see the appearance of Riccati solutions for $n=0$ in terms of the surfaces: the $n=0$ specialisation causes the points $q_2$, $q_3$ to move onto the line $\{ y=0 \}$ so the class $\h_y-\F_2-\F_3$ becomes effective, with the Riccati solution given by the flow along this curve. 
This is the class in $\operatorname{Pic}(\mathcal{X}_{\mathbf{b}})$ of a \emph{nodal curve} : a $(-2)$-curve which is not a component of the anti-canonical divisor \cite{Sak:2001:RSAWARSGPE}.
\end{remark}


\section{Conclusions} 
\label{sec:conclusions}
The purpose of this paper is to provide a novel method for solving the identification problem for differential Painlev\'e equations, i.e., a systematic procedure for determining whether a second-order non-linear differential system  
can be transformed to a differental Painlev\'e equation, and if so, how to reduce it to the standard form. 
The method is essentially justified by the uniqueness results related to global Hamiltonian systems on Okamoto's spaces, so an appropriate identification obtained through the geometric approach is guaranteed to match the differential systems.

We considered in detail two examples from the theory of 
continuous and discrete orthogonal polynomials, where we showed that the recurrence coefficients for these polynomials are expressed in terms of  solutions of the standard fourth and sixth  Painlev\'e equations. 
In particular the transformations both show features which would likely be difficult to detect through a brute force approach, for example a nontrivial relation of independent variables in \autoref{changeofvarslaguerre}, and quite complicated forms of the birational transformations in \autoref{theorem P6 final} and \autoref{theorem P6 initial}.
While we have presented examples from the theory of orthogonal polynomials, it is worth noting that this method is applicable to a wide range of other applied problems where differential Painlev\'e equations are either known or suspected to appear.

\section*{Acknowledgements} 
\label{sec:acknowledgements} AD acknowledges the support of 
 the MIMUW grant  to visit Warsaw in January 2020. GF acknowledges the support of the 
National Science Center (Poland) via grant OPUS 2017/25/B/BST1/00931. 
Part of this work was carried out while AS was supported by a University College London Graduate Research Scholarship and Overseas Research Scholarship. 
AS also acknowledges the support of the MIMUW grant which enabled him to  visit Warsaw in December 2019 and February 2020, which was essential for the success of the project.
AS was supported by a London Mathematical Society Early Career Fellowship during the preparation of this manuscript and gratefully acknowledges the support of the London Mathematical Society.
The authors also thank an anonymous reviewer for valuable comments and suggestions, and for pointing out the alternative method for detecting the connection between equation \eqref{differential} to the fourth Painlev\'e equation as outlined in \autoref{Laguerresubsubsection}.




\begin{appendices}

\section{Surfaces from the hypergeometric weight} \label{AppendixA}


Here we present the surfaces constructed in \cite{DFS19}, which provide the space of initial conditions for the differential system \eqref{syst P6} from the hypergeometric weight.
As usual we begin with the affine $(x,y)$-chart and compactifying to $\p^1 \times \p^1$ via $X=1/x$, $Y=1/y$. 
In the $(x,y)$-chart we blow up four points $q_1,\,q_2,\,q_3$ and $q_4$, which lie on the curve $\Gamma$ on $\p^1 \times \p^1$ defined by the equation
\begin{equation}\label{curve}
\Gamma ~:~ s(x,\,y)=x^2-(n+\alpha+\beta)x-y+\alpha\beta=0,
\end{equation}
which can be seen to be the vanishing locus of the denominators of both $x'(c)$ and $y'(c)$ in the system \eqref{syst P6}.
In the $(X,Y)$-chart we find a cascade of four points over $q_5 : (X,Y) = (0,0)$. 
We give the configuration of these points in \autoref{fig:surface-hw}, with their locations in coordinates in \autoref{pointlocationshypergeometric}. 
The inaccessible divisors are 
\begin{equation} \label{DiHypergeometric}
\begin{aligned}
D_0  &= F_5-F_6, 				&\quad  	&D_3 = F_7 - F_8, \\
D_1 &= 2 H_x +H_y - F_{123456}, 	&\quad 	&D_4 = H_y - F_{56}, \\
D_2 &= F_6 - F_7, 				&\quad 	&
\end{aligned}
\end{equation}
where we have used notation $F_{ij} = F_i + F_j$ etc. for brevity, 
and in particular the divisor $D_1$ is the proper transform of the curve $\Gamma$ defined by \eqref{curve}.
The classes of these inaccessible divisors $\delta_i = [D_i]$ in $\Pic(\X_{c})$ provide the surface root basis in \autoref{surfacerootshypergeom}. 
These form the irreducible components of the unique anticanonical divisor of the surface $\X_c$, with
\begin{equation}
\begin{aligned}
-\mathcal{K}_{\X_c} &= \delta_0 + \delta_1 + 2 \delta_2 + \delta_3 + \delta_4 \\
&= 2 \h_x + 2 \h_y - \F_1 - \F_2 - \F_3 - \F_4 - \F_5 - \F_6 - \F_7 - \F_8.
\end{aligned}
\end{equation}
We take the rational symplectic form on $\X_c$ given in charts by 
\begin{equation} \label{omegahypergeom}
\begin{aligned}
\omega &= k \frac{dy \wedge dx}{x^2 - (n + \alpha + \beta)x - y + \alpha \beta}  \\
&= k \frac{dy \wedge dX}{ X^2 y - \alpha \beta X^2 +(n + \alpha + \beta) X - 1}\\
&= k \frac{dx \wedge dY}{Y \left( x^2 Y - (n+\alpha + \beta)x Y + \alpha \beta Y - 1 \right)}   \\
&= k \frac{dX \wedge dY}{Y(X^2 - \alpha \beta X^2 Y+(n+\alpha + \beta)X Y - Y)}, \\
\end{aligned}
\end{equation}
where $k$ is normalised appropriately, for example to $k=1$ for the change of variables from the initial identification, provided in Lemma \ref{rootvarsinit}.

\begin{figure}[H]
	\begin{tikzpicture}[scale=1,>=stealth,basept/.style={circle, draw=red!100, fill=red!100, thick, inner sep=0pt,minimum size=1.2mm}]
			\begin{scope}[xshift = -4cm]
			\draw [black, line width = 1pt]  	(4,0) 	-- (-0.5,0) 		node [left] {$y=0$} node[pos=0, right] {};
			\draw [black, line width = 1pt] 	(4,2.5) 	-- (-0.5,2.5)	node [left] {$y=\infty$} node[pos=0, right] {};
			\draw [black, line width = 1pt] 	(0,3) -- (0,-0.5)			node [below] {$x=0$} node[pos=0, above, xshift=-7pt] {};
			\draw [black, line width = 1pt] 	(3,3) -- (3,-0.5)			node [below] {$x=\infty$} node[pos=0, above, xshift=7pt] {};

			\draw[line width = 1pt] (3,1.5) circle [x radius=1.5cm, y radius = 1cm];
			\node at (3,0.5) 	[circle, draw=white!100, fill=white!100, minimum size=1mm] {};
			\draw [black, line width = 1pt] 	(3,3) -- (3,-0.5)			node [below] {};

			\node (p1) at (1.9,2.18) [basept,label={[xshift=10pt, yshift = -15 pt] $q_{1}$}] {};
			\node (p2) at (1.6,1.14) [basept,label={[xshift=-10pt, yshift = -10 pt] $q_{2}$}] {};
			\node (p3) at (2.1,0.7) [basept,label={[xshift=0pt, yshift = -20 pt] $q_{3}$}] {};
			\node (p4) at (3.4,0.54) [basept,label={[xshift=10pt, yshift = -15 pt] $q_{4}$}] {};
			\node (p5) at (3,2.5) 	[basept,label={[xshift=10pt, yshift = 0 pt] $q_{5}$}] {};
			\node (p6) at (2.1,3.1)	[basept,label={[above] 		$q_{6}$}] {};
			\node (p7) at (1.4,3.1) [basept,label={[above] 		$q_{7}$}] {};
			\node (p8) at (0.7,3.1) [basept,label={[above] 		$q_{8}$}] {};

			\node at (4.7,0.5) [label = {\tiny}] {};
			\node at (5.3,1.6) [label = {\tiny }] {};

			\draw [line width = 0.8pt, ->]  (p8) edge (p7) (p7) edge (p6) (p6) edge (p5);
		\end{scope}

		\draw [->] (2.5,1.5)--(1.5,1.5) node[pos=0.5, below] {\small$\operatorname{Bl}_{q_{1}\cdots q_{8}}$};

		\begin{scope}[xshift = 4cm, yshift = 0cm]
			\draw [black, line width = 1pt ] (3.5,0) -- (-0.5,0) node [left] {$$};
			\draw [blue, line width = 1pt ] (1.2,3.7) .. controls (0.8,3) and (0.3,2.5) .. (-0.5,2.5) node [left] {\small$D_5$};
			\draw [black, line width = 1pt ] (0,3) -- (0,-0.5) node [below] {$$};
			\draw [red, line width = 1pt ] (4.5,1) .. controls (4,0.8) and (3,0.3) .. (3,-0.5) node [below] {\small$H_{x} - F_{5}$};

			\draw[blue, line width = 1pt] (3.5,1.9) circle [x radius=1.7cm, y radius = 1cm, rotate = 25] ;
			\node at (4,1.05) 	[circle, draw=white!100, fill=white!100, minimum size=1mm] {};

			\node at (4,3.4) [blue] {\small $D_1$};

			\draw [red, line width = 1 pt] (1.6,1.9) -- (2.5,2.1) node [pos = 0, left] {\small$F_{1}$};
			\draw [red, line width = 1 pt] (1.6,1.1) -- (2.5,1.6) node [pos = 0, left] {\small $F_{2}$};
			\draw [red, line width = 1 pt] (1.95,0.5) -- (2.65,1.2) node [pos = 0, left] {\small $F_{3}$};
			\draw [red, line width = 1 pt] (4.5,2.5) -- (5.4,2.3) node [right] {\small $F_{4}$};

			\draw [blue, line width = 1 pt] (0.4,3.4) -- (4.2,1.8) node [pos = 0, left] {\small $D_2$};
			\draw [blue, line width = 1 pt] (3.5,2.5) .. controls (3.4,1.8) and (3.8,1.2) .. (4.5,0.5) node [below] {\small$D_0$};
			\draw [blue, line width = 1 pt] (0.8,2.3)--(1.8,4) node [pos = 1, right] {\small$D_3$};
			\draw [red, line width = 1 pt] (1.1,4) -- (2,3.6) node [pos = 0, left] {\small$F_{8}$};


		\end{scope}


	\end{tikzpicture}
	\caption{Point configuration and surface for the hypergeometric weight}
	\label{fig:surface-hw}
\end{figure}
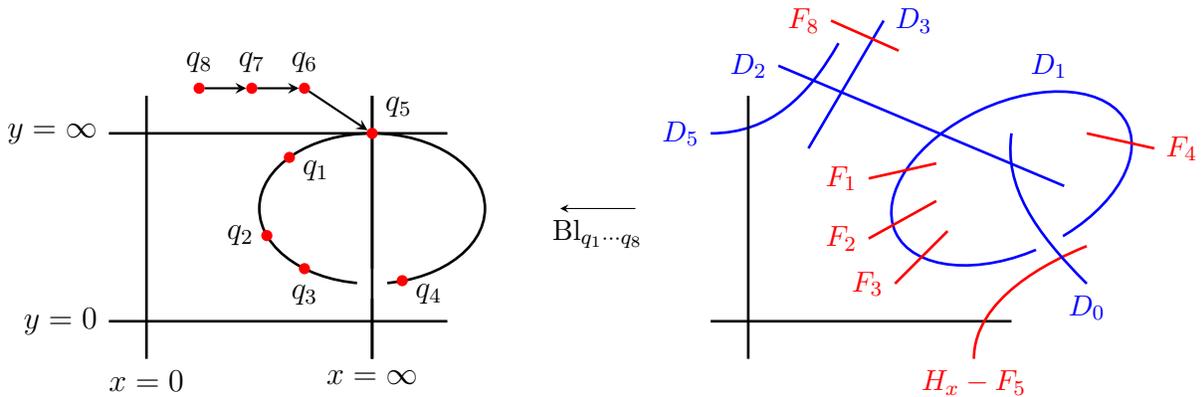

\begin{figure}[ht]
\begin{gather*}
\begin{aligned}
& q_1 : (x, y) = (1,(\alpha-1)(\beta-1) - n), \quad q_2 : (x,y) = (\alpha, - n \alpha),  \\
\\
& q_3 : (x, y) = (\beta, - n \beta), \quad q_4 : (x, y) = (\gamma, \alpha \beta - (n+\alpha + \beta)\gamma + \gamma^2 ), 
\end{aligned}
\\
\\
\begin{aligned}
& q_5 : (x,y) = (\infty, \infty) ~ \leftarrow  &~   &q_6 : (U_5, V_5) = (x/y, 1/x) = (0,0) \\
&  & &  \uparrow  & \\
&  & & q_7 : (u_6,v_6) = \left( U_5/V_5 , V_5 \right) = (c/(c-1), 0)  \\
&  & &  \uparrow  & \\
&  & & q_8 : (u_7,v_7) =\left( \frac{ (c-1) u_6 - c}{(c-1)v_6} , v_6 \right) \\
&  & & \quad \quad ~~~~~~~~~ = \left( \frac{c((c+1)n+(\alpha+\beta)c-\gamma)}{(c-1)^2},  0 \right) 
\end{aligned}
\end{gather*}

\caption{Point locations for the hypergeometric weight}
 \label{pointlocationshypergeometric}
 \end{figure}
 
 \begin{figure}[ht]
 \begin{equation*}  
	\raisebox{-32.1pt}{\begin{tikzpicture}[
			elt/.style={circle,draw=black!100,thick, inner sep=0pt,minimum size=2mm}, scale=1]
		\path 	(-1,1) 	node 	(d0) [elt, label={[xshift=-10pt, yshift = -10 pt] $\delta_{0}$} ] {}
		        (-1,-1) node 	(d1) [elt, label={[xshift=-10pt, yshift = -10 pt] $\delta_{1}$} ] {}
		        ( 0,0) 	node  	(d2) [elt, label={[xshift=13pt, yshift = -12 pt] $\delta_{2}$} ] {}
		        ( 1,1) 	node  	(d3) [elt, label={[xshift=10pt, yshift = -10 pt] $\delta_{3}$} ] {}
		        ( 1,-1) node 	(d4) [elt, label={[xshift=10pt, yshift = -10 pt] $\delta_{4}$} ] {};
		\draw [black,line width=1pt ] (d0) -- (d2) -- (d1)  (d3) -- (d2) -- (d4);
	\end{tikzpicture}} \qquad
			\begin{alignedat}{2}
			\delta_{0} &= \mathcal{F}_{5} - \mathcal{F}_{6}, &\qquad  \delta_{3} &= \mathcal{F}_{7} - \mathcal{F}_{8},\\
			\delta_{1} &= 2\mathcal{H}_{x} + \mathcal{H}_{y} - \mathcal{F}_{123456}, &\qquad  \delta_{4} &= \mathcal{H}_{y} - \mathcal{F}_{56}.\\
			\delta_{2} &= \mathcal{F}_{6} - \mathcal{F}_{7},
			\end{alignedat}
\end{equation*}
\caption{The Surface Root Basis for the hypergeometric weight}
 \label{surfacerootshypergeom}
 \end{figure}
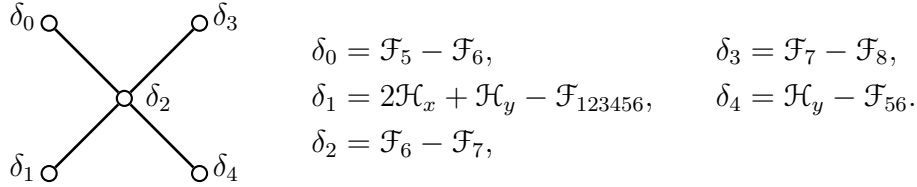


\section{Standard model of $D_4^{(1)}$-surfaces and $\pain{VI}$} \label{AppendixB}

We also recall the standard model of Sakai surfaces of surface type $D_4^{(1)}$, as well as the standard form of $\pain{VI}$ as a first-order Hamiltonian system.
The usual Hamiltonian form of $\pain{VI}$ \cite{KajNouYam:2017:GAOPE} is given by
\begin{equation} \label{hamP6}
\frac{dq}{dt} = \frac{\partial H}{\partial p}, \qquad  \frac{dp}{dt} = -\frac{\partial H}{\partial q} ,
\end{equation}
where the Hamiltonian is given by 
\begin{equation*}
H(q,p,t) = \frac{q(q-1)(q-t)}{t(t-1)} \left( p^2 - \left( \frac{a_0 -1}{q-t} + \frac{a_3}{q-1} + \frac{a_4}{q} \right) p \right) + \frac{ (q-t)a_2 (a_1 + a_2)}{t(t-1)}.
\end{equation*}
Here $a_0,a_1,a_2,a_3, a_4$ are parameters subject to the normalisation 
\begin{equation} \label{paramnormP6}
a_0 + a_1 + 2 a_2 + a_3+ a_4  = 1.
\end{equation}
Eliminating $p$ from the system \eqref{hamP6}, we obtain the usual scalar form of $\pain{VI}$ for $q(t)$:
\begin{equation} \label{scalarP6}
\begin{aligned}
\frac{d^2 q}{d t^2} &= \frac{1}{2} \left( \frac{1}{q} + \frac{1}{q-1} + \frac{1}{q-t} \right) \left( \frac{d q}{dt} \right)^2 - \left( \frac{1}{t} + \frac{1}{t-1} + \frac{1}{q-t} \right) \left( \frac{dq}{dt} \right) \\
&\quad \quad \quad \quad \quad \quad + \frac{q (q-1)(q-t)}{t^2(t-1)^2} \left( A + B \frac{t}{q^2} + C \frac{t-1}{(y-1)^2} + D \frac{t(t-1)}{(y-t)^2} \right),
\end{aligned}
\end{equation}
where the parameters are given by
\begin{equation}
A=\frac{a_1^2}{2},\;\;B=-\frac{a_4^2}{2},\;\;C=\frac{a_3^2}{2},\;\;D=\frac{1 - a_0^2}{2}.
\end{equation}
The standard form of $\pain{VI}$ (as a first-order system) we will be working with is obtained from \eqref{hamP6} by letting $(f,g) = (q, q p)$, after which we have 
\begin{equation} \label{fgP6}
\frac{f'}{f} = \frac{\partial H_1}{\partial g}, \quad \frac{g'}{f} = - \frac{\partial H_1}{\partial f}, 
\end{equation}
where the Hamiltonian is given by  
\begin{equation*}\label{Ham fg}
H_1(f,g,t)=\frac{a_2(a_1+a_2)(f-t)}{t(t-1)}+\frac{f(f-1)(f-t)}{t(t-1)}\left\{ \frac{g^2}{f}-\left( \frac{a_4}{f}+\frac{a_3}{f-1}+\frac{a_0-1}{f-t}\right) \right\}.
\end{equation*}

\subsection{Point configuration and anticanonical divisor}

Beginning with the affine $(f,g)$-chart and compactifying to $\p^1 \times \p^1$ as usual via $F= 1/f$ and $G=1/g$, the configuration of points to be blown up to construct the surface, which we denote $\tilde{\X}_t$, are given in \autoref{surfaceP6}, with their locations in coordinates in \autoref{pointlocationsP6}.


\begin{figure}[ht]
	\begin{tikzpicture}[>=stealth,basept/.style={circle, draw=red!100, fill=red!100, thick, inner sep=0pt,minimum size=1.2mm}]
		\begin{scope}[xshift = -1cm]
			\draw [black, line width = 1pt] 	(4.1,2.5) 	-- (-0.5,2.5)	node [left]  {$g=\infty$} node[pos=0, right] {$$};
			\draw [black, line width = 1pt] 	(0,3) -- (0,-1)			node [below] {$f=0$}  ;
			\draw [black, line width = 1pt] 	(3.6,3) -- (3.6,-1)		node [below]  {$f=\infty$} ;
			\draw [black, line width = 1pt] 	(4.1,-.5) 	-- (-0.5,-0.5)	node [left]  {$g=0$} node[pos=0, right] {$$};

			\node (p1) at (3.6,0.5) [basept,label={[xshift=10pt, yshift = -10 pt] $p_{1}$}] {};
			\node (p2) at (3.6,1.5) [basept,label={[xshift=10pt, yshift = -10 pt] $p_{2}$}] {};
			\node (p3) at (2.4,2.5) [basept,label={[above] $p_{3}$}] {};
			\node (p4) at (2.4,1.5) [basept,label={[xshift=0pt, yshift = -18 pt] $p_{4}$}] {};
			\node (p5) at (0,0.5) [basept,label={[xshift=-10pt, yshift = -10 pt] $p_{5}$}] {};
			\node (p6) at (0,1.5) [basept,label={[xshift=-10pt, yshift = -10 pt] $p_{6}$}] {};
			\node (p7) at (1.2,2.5) [basept,label={[above] $p_{7}$}] {};
			\node (p8) at (1.2,1.5) [basept,label={[xshift=0pt, yshift = -18 pt] $p_{8}$}] {};
			\draw [line width = 0.8pt, ->] (p4) -- (p3);
			\draw [line width = 0.8pt, ->] (p8) -- (p7);
		\end{scope}
	
		\draw [->] (6.0,1.5)--(4.0,1.5) node[pos=0.5, below] {$\text{Bl}_{p_1\dots p_8}$};
	
		\begin{scope}[xshift = 8cm]
			\draw [blue, line width = 1pt] 	(4.1,2.5) 	-- (-0.5,2.5)	node [left] {}node[pos=0, right] {$\tilde{D}_2$};
			\draw [blue, line width = 1pt] 	(0,3) -- (0,-.75)			node [below] {$\tilde{D}_4$} node[pos=0, above, xshift=-7pt] {} ;
			\draw [blue, line width = 1pt] 	(3.6,3) -- (3.6,-0.75)			node [below] {$\tilde{D}_1$}  node[pos=0, above, xshift=7pt] {};
		
			\draw [red,line width = 1pt] (3.1,0.7) -- (4.1,0.3) node [right] {$E_1$};
			\draw [red,line width = 1pt] (3.1,1.7) -- (4.1,1.3) node [right] {$E_2$};
			\draw [red,line width = 1pt] (-0.5,0.3) -- (0.5,0.7) node [pos=0,left] {$E_5$} ;
			\draw [red,line width = 1pt] (-0.5,1.3) -- (0.5,1.7) node [pos=0,left] {$E_6$};
			\draw [blue,line width = 1pt] (2.4,0.7) -- (2.4,2.9) node [pos=0,below] {$\tilde{D}_0$} ;
			\draw [blue,line width = 1pt] (1.2,0.7) -- (1.2,2.9) node [pos=0,below]  {$\tilde{D}_3$};
			\draw [red,line width = 1pt] (1.9,1.5) -- (2.9,1.1) node [right] {$E_4$};
			\draw [red,line width = 1pt] (0.7,1.1) -- (1.7,1.5) node [pos=0,left] {$E_8$};

		\end{scope}
	\end{tikzpicture}
	\caption{Point configuration for the standard model of $D_4^{(1)}$-surfaces}
	\label{surfaceP6}
\end{figure}
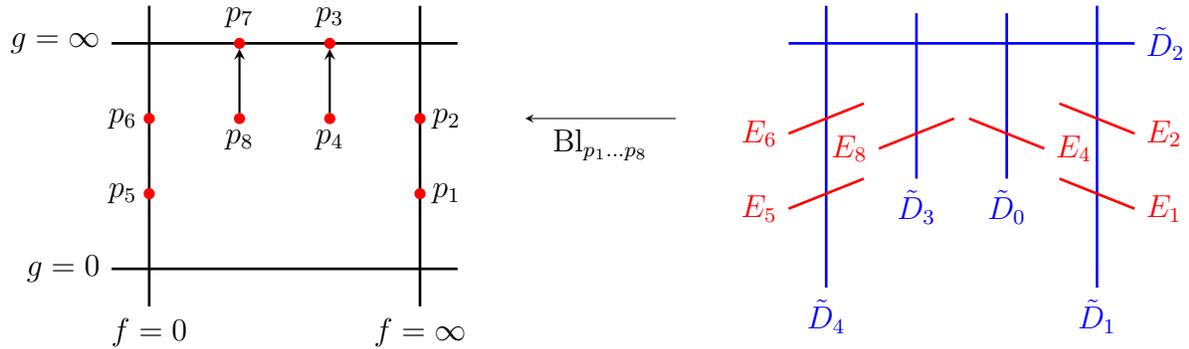

\begin{figure}[ht]
\begin{equation*}
\begin{aligned}
&  p_1 : (f,g) = (\infty, -a_2),  &&p_2 : (f,g) = ( \infty, -a_1 -a_2), \\
\\
& p_3 : (f,g) = (t , \infty) \qquad \leftarrow &&p_4 : (u_3, v_3) = ((f-t)g, 1/g) = (a_0 t,0),\\
\\
&  p_5 : (f,g) = (0,0), &&p_6 : (f,g) = (0, a_4), \\
\\
& q_7 : (f,g) = (1,0) \qquad ~ \leftarrow  &&p_8 : (u_7, v_7) = \left( (f-1)g, 1/g \right) = (a_3,0).
\end{aligned}
\end{equation*}
\caption{Point locations for the standard model of $D_4$-surfaces}
 \label{pointlocationsP6}
 \end{figure}
The inaccessible divisors are 
\begin{equation} \label{DiP6}
\begin{aligned}
\tilde{D}_0  &= E_3 - E_4, 		&\quad  	&\tilde{D}_3 = E_7 - E_8, \\
\tilde{D}_1 &= H_f - E_1 - E_2, 		&\quad 	&\tilde{D}_4 = H_f - E_5 - E_6, \\
\tilde{D}_2 &= H_g - E_3 - E_7, 	&\quad 	&
\end{aligned}
\end{equation}
where in particular $\tilde{D}_1$ and $\tilde{D}_4$ are the proper transforms of the lines $\{F=0\}$ and $\{f=0\}$ respectively, and similarly $\tilde{D}_2$ is the proper transform of the line $\{G=0\}$. 
The classes $\tilde{\delta}_i = [\tilde{D}_i]$ in $\Pic(\tilde{\X}_{t})$ provide the surface root basis in \autoref{surfacerootsP6}. 
The inaccessible divisors in \eqref{DiP6} are the irreducible components of the unique anticanonical divisor of the surface $\tilde{\X}_t$, with
\begin{equation}
\begin{aligned}
-\mathcal{K}_{\tilde{\X}_t} &= \tilde{\delta}_0 + \tilde{\delta}_1 + 2 \tilde{\delta}_2 + \tilde{\delta}_3 + \tilde{\delta}_4 \\
&= 2 \h_f + 2 \h_g - \E_1 - \E_2 - \E_3 - \E_4 - \E_5 - \E_6 - \E_7 - \E_8.
\end{aligned}
\end{equation}

\begin{figure}[ht]
\begin{equation*}		
	\raisebox{-32.1pt}{\begin{tikzpicture}[
			elt/.style={circle,draw=black!100,thick, inner sep=0pt,minimum size=2mm},scale=1]
		\path 	(-1,1) 	node 	(d0) [elt, label={[xshift=-10pt, yshift = -10 pt] $\tilde{\delta}_{0}$} ] {}
		        (-1,-1) node 	(d1) [elt, label={[xshift=-10pt, yshift = -10 pt] $\tilde{\delta}_{1}$} ] {}
		        ( 0,0) 	node  	(d2) [elt, label={[xshift=13pt, yshift = -12 pt] $\tilde{\delta}_{2}$} ] {}
		        ( 1,1) 	node  	(d3) [elt, label={[xshift=10pt, yshift = -10 pt] $\tilde{\delta}_{3}$} ] {}
		        ( 1,-1) node 	(d4) [elt, label={[xshift=10pt, yshift = -10 pt] $\tilde{\delta}_{4}$} ] {};
		\draw [black,line width=1pt ] (d0) -- (d2) -- (d1)  (d3) -- (d2) -- (d4); 
	\end{tikzpicture}} \qquad
			\begin{alignedat}{2}			
			\tilde{\delta}_{0} &= \mathcal{E}_{3} - \mathcal{E}_{4}, &\qquad  \tilde{\delta}_{3} &= \mathcal{E}_{7} - \mathcal{E}_{8},\\
			\tilde{\delta}_{1} &= \mathcal{H}_{f} - \mathcal{E}_{1} - \mathcal{E}_{2}, &\qquad  \tilde{\delta}_{4} &= \mathcal{H}_{f} - \mathcal{E}_{5} - \mathcal{E}_{6}.\\
			\tilde{\delta}_{2} &= \mathcal{H}_{g} - \mathcal{E}_{3} - \mathcal{E}_{7},\\[5pt]
			\end{alignedat}
\end{equation*}
\caption{The Surface Root Basis for the standard model of $D_4^{(1)}$-surfaces}
 \label{surfacerootsP6}
 \end{figure}
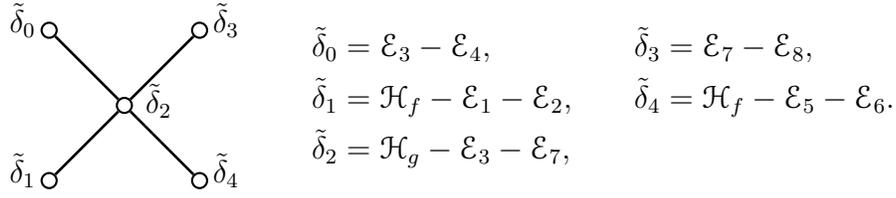

\subsection{Symmetry root basis and root variables}

We make the same choice of symmetry root basis as in \cite{KajNouYam:2017:GAOPE}, which we present in \autoref{symmetryrootsP6}.

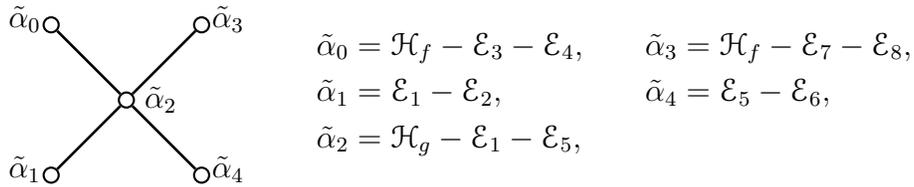
\begin{figure}[H]
\begin{equation*}		
	\raisebox{-32.1pt}{\begin{tikzpicture}[
			elt/.style={circle,draw=black!100,thick, inner sep=0pt,minimum size=2mm},scale=1]
		\path 	(-1,1) 	node 	(d0) [elt, label={[xshift=-10pt, yshift = -10 pt] $\tilde{\alpha}_{0}$} ] {}
		        (-1,-1) node 	(d1) [elt, label={[xshift=-10pt, yshift = -10 pt] $\tilde{\alpha}_{1}$} ] {}
		        ( 0,0) 	node  	(d2) [elt, label={[xshift=13pt, yshift = -12 pt] $\tilde{\alpha}_{2}$} ] {}
		        ( 1,1) 	node  	(d3) [elt, label={[xshift=10pt, yshift = -10 pt] $\tilde{\alpha}_{3}$} ] {}
		        ( 1,-1) node 	(d4) [elt, label={[xshift=10pt, yshift = -10 pt] $\tilde{\alpha}_{4}$} ] {};
		\draw [black,line width=1pt ] (d0) -- (d2) -- (d1)  (d3) -- (d2) -- (d4); 
	\end{tikzpicture}} \qquad
			\begin{alignedat}{2}			
			\tilde{\alpha}_{0} &= \h_f - \mathcal{E}_{3} - \mathcal{E}_{4}, &\qquad  \tilde{\alpha}_{3} &= \h_f - \mathcal{E}_{7} - \mathcal{E}_{8},\\
			\tilde{\alpha}_{1} &= \mathcal{E}_{1} - \mathcal{E}_{2}, &\qquad  \tilde{\alpha}_{4} &= \mathcal{E}_{5} - \mathcal{E}_{6},\\
			\tilde{\alpha}_{2} &= \mathcal{H}_{g} - \mathcal{E}_{1} - \mathcal{E}_{5},\\[5pt]
			\end{alignedat}
\end{equation*}
\caption{The Symmetry Root Basis for the standard model of $D_4^{(1)}$-surfaces}
 \label{symmetryrootsP6}
 \end{figure}
To define the root variables corresponding to this symmetry root basis for the surface $\tilde{\X}_t$, we use the symplectic form whose pole divisor is the anticanonical divisor of $\tilde{\X}_t$, which is defined in charts by
\begin{equation} \label{omegaP6}
\tilde{\omega} = \tilde{k} \frac{ d g \wedge df }{f}=  - \tilde{k} \frac{ d G \wedge df}{G^2 f}  =- \tilde{k} \frac{ d g \wedge dF}{F} = k \frac{d G \wedge d F}{G^2 F},
\end{equation}
where again $\tilde{k}$ is a nonzero constant which will be normalised. 
By standard computations (see \cite{DFS19} for details) we have the following.

\begin{lemma}

\begin{enumerate}[(a)]
\item The values of the period mapping for the symplectic form $\tilde{\omega}$ in \eqref{omegaP6} on the symmetry roots in \autoref{symmetryrootsP6} are given by
\begin{equation}
\tilde{\chi}(\tilde{\alpha}_0) = \tilde{k} a_0, \quad  \tilde{\chi}(\tilde{\alpha}_1) = \tilde{k} a_1, \quad \tilde{\chi}(\tilde{\alpha}_2) = \tilde{k} a_2, \quad \tilde{\chi}(\tilde{\alpha}_3) = \tilde{k} a_3, \quad \tilde{\chi}(\tilde{\alpha}_4) = \tilde{k} a_4.
\end{equation}
\item By normalising the symplectic form such that $\tilde{k}=1$, we have the parameters $a_i$ in the point configuration being the root variables for the surface $\tilde{\X}_t$ with the choice of symmetry root basis in \autoref{symmetryrootsP6}.
This also recovers the parameter normalisation \eqref{paramnormP6}:
\begin{equation}
a_0 + a_1 +2 a_2 +a_3 + a_4= \tilde{\chi}(\tilde{\alpha}_0 + \tilde{\alpha}_1 + 2\tilde{\alpha}_2 + \tilde{\alpha}_3 + \tilde{\alpha}_4) = \tilde{\chi}(-\mathcal{K}_{\tilde{\X}_t}) = 1.
\end{equation}
\end{enumerate}

\end{lemma}

\end{appendices}


\bibliographystyle{amsalpha}

\end{document}